\documentclass[11pt,reqno]{amsart}
\usepackage{amsmath}
\usepackage{amssymb}
\usepackage{amsthm}
\usepackage{enumerate}
\usepackage[usenames]{color}
\usepackage{eepic,epic}
\usepackage{url} 
\usepackage{makecell}
\usepackage{epstopdf}

 \usepackage{algorithm,algorithmic}

\usepackage{multirow}
\usepackage{graphicx}
\usepackage{kotex}

\newtheorem{thm}{Theorem}[section]

\newtheorem{lem}[thm]{Lemma}
\newtheorem{prop}[thm]{Proposition}

\theoremstyle{definition}
\newtheorem{defn}[thm]{Definition}

\newtheorem{ass}{Assumption}
\theoremstyle{remark}
\newtheorem{rem}[thm]{Remark}
\numberwithin{equation}{section}

\newcommand{\ep}{\epsilon}

\begin{document}

\title[]
{Convergence result for the gradient-push algorithm  and its application to boost up the Push-DIging algorithm}

\author{Hyogi Choi}
\address{Department of Mathematics, Sungkyunkwan University, Suwon 440-746, Republic of Korea}
\email{dyddmlsns1@skku.edu}

\author{Woocheol Choi}
\address{Department of Mathematics, Sungkyunkwan University, Suwon 440-746, Republic of Korea}
\email{choiwc@skku.edu}

\author{Gwangil Kim}
\address{Department of Mathematics, Sungkyunkwan University, Suwon 440-746, Republic of Korea}
\email{nice1708@skku.edu}
 
\subjclass[2010]{Primary  }


\maketitle

\begin{abstract} The gradient-push algorithm is a fundamental algorithm for the distributed optimization problem 
\begin{equation}
\min_{x \in \mathbb{R}^d} f(x) = \sum_{j=1}^n f_j (x),
\end{equation}
where each local cost $f_j$ is only known to agent $a_i$ for $1 \leq i \leq n$ and the agents are connected by a directed graph. In this paper, we obtain convergence results for the gradient-push algorithm with constant stepsize whose range is sharp in terms the order of the smoothness constant $L>0$. Precisely, under the two settings: 1) Each local cost $f_i$ is strongly convex and $L$-smooth, 2) Each local cost $f_i$ is convex quadratic and $L$-smooth while the aggregate cost $f$ is strongly convex, we show that the gradient-push algorithm with stepsize $\alpha>0$ converges to an $O(\alpha)$-neighborhood of the minimizer of $f$ for a range $\alpha \in (0, c/L]$ with a value $c>0$ independent of $L>0$. As a benefit of the result, we suggest a hybrid algorithm that performs the gradient-push algorithm with a relatively large stepsize $\alpha>0$ for a number of iterations and then go over to perform the Push-DIGing algorithm. It is verified by a numerical test that the hybrid algorithm enhances the performance of the Push-DIGing algorithm significantly. The convergence results of the gradient-push algorithm are also supported by numerical tests.
\end{abstract}


\section{Introduction}\label{sec-1}
In this paper, we consider the distributed optimization problem:
\begin{equation} \label{eq-1-1}
\min_{x \in \mathbb{R}^d} f(x) = \sum_{j=1}^n f_j(x),
\end{equation} 
where $n$ denotes the number of agents in the network, and $f_i: \mathbb{R}^d \rightarrow \mathbb{R}$ represents the local objective function available only to agent $i$ for each $1 \leq i \leq n$. {Each agent can communicate with its neighboring agents in the network, which can be represented as a graph.} In this work, we consider the case that $f$ is stronlgy convex, and so there exists a unique minimizer $x_*$ of the problem \eqref{eq-1-1}.

Distributed optimization has been widely studied recently since it appears in various applications such as wireless sensor networks \cite{BG, SGRR}, multi-agent control \cite{MC1, MC2, MC3}, smart grids \cite{SG1, SG2}, and machine learning \cite{ML1, ML2, ML3, ML4, ML5}. We refer to  \cite{NO1,   NOS, QL, SLWY, XK2} and references therein regarding various algorithms for distributed optimization.


Frequently in real-world scenarios, the communication between agents is described as a directed network because the communication powers of agents are different, leading to one-way interactions between agents.
In the fundamental work,  Nedic {and} Olshevsky \cite{NO1, NO2} introduced the gradient-push algorithm for the distributed optimization on the directed graph making use of the gradient-push protocol  \cite{KDG, TLR2}. The gradient-push algorithm is formulated as follows: 
\begin{equation} \label{eq-1-2}
\begin{split}
&w_i(t+1) = \sum_{j=1}^n W_{ij}x_j(t)
\\
&y_i(t+1) = \sum_{j=1}^n W_{ij}y_j(t)
\\
&z_i(t+1) = \frac{w_i(t+1)}{y_i(t+1)}
\\
&x_i(t+1) =w_i(t+1) -\alpha(t)\nabla f_i\big(z_i(t+1)\big),
\end{split}
\end{equation}
endowed with initial data $x_i(0) \in \mathbb{R}^d$ and $y_i(0)=1$ for each $1 \leq i \leq n$ and $t\geq 0$. Here $W_{ij} \geq 0$ is a communication weight and $\alpha (t) >0$ is stepsize. We define a directed graph $G = (V, E)$ with the vertex set $V = {1, 2, \cdots , n}$ and the edge set $E \subset V \times V$ , where $(i, j) \in E$ means that there exist a directed edge from agent $j$ toward $i$. We assume that $W_{ij}>0$ if $(i,j) \in E$ or $i=j$ and $W_{ij}=0$ otherwise. A popular choice is given by $W_{ij} =1/(|N_j|+1)$ for $i \in N_j \cup \{j\}$ where $N_j = \{i|(i,j) \in E\}$.  Throughout this paper, we will focus on the constant stepsize $\alpha(t)=\alpha$ and the graph $G$ is assumed to satisfy the following assumptions.

\begin{ass}\label{ass-1} Graph $G$ is strongly connected.
 When $G$ is strongly connected, then it contains directed edge $(i,j) \in V \times V $ for all $i,j$ 
\end{ass} 
\begin{ass}\label{ass-2}the matrix $W = \{W_{ij}\}_{1 \leq i,j \leq n}$ is column stochastic, which means
\begin{equation*}
\sum_{k=1}^nW_{kj}=1\label{eq-1-8}.
\end{equation*}
\end{ass}
\begin{ass} Each local funcition $f_i$ satisfies $L_i$-smoothness for some $L_i >0$, i.e., 
\begin{equation*}
\|\nabla f_i (x) - \nabla f_i (y) \|\leq L_i \|x-y\| \quad \forall~x,y \in \mathbb{R}^d.
\end{equation*}
\end{ass}
 

   
 The gradient-push algorithm was extensively studied by many researchers for various settings such as online distributed optimization \cite{AGL, LGW,O3}, the gradient-push algorithm involving the quantized communications \cite{Q2}, zeroth-order distributed optimization \cite{GF}, and the gradient-push algorithm with the event-{triggered} communication  \cite{CK2}. The work \cite{C1} applied the gradient-push algorithm for distributed optimization problem with constraints. The work \cite{AGP} studied the asynchronous version of the gradient-push algorithm.

\subsection{New results for the gradient-push algorithm with constant stepsize}
Despite its wide applications, the convergence property of the gradient-push algorithm has not been fully understood. 
The works \cite{NO1, NO3, NO2, TT} studied the convergence property of the gradient-push algorithm after assuming the boundedness of either the state variable or the gradient for the convergence analysis. 
 The work \cite{NO2} established the convergence result for convex cost functions when the stepsize $\alpha (t)$ is given as $\alpha (t) = 1/\sqrt{t}$ or satisfies the condition $\sum_{t=1}^{\infty}\alpha (t) =\infty$ and $\sum_{t=1}^{\infty}\alpha (t)^2 <\infty$. The work \cite{NO3} obtained the convergence result when the cost function is strongly convex, and $\alpha (t) = c/t$ for some $c>0$. We also refer to the work \cite{TT} where the authors obtained a convergence result for non-convex cost functions. We would like to mention that the convergence results in the previous works \cite{NO1, NO2, TT} rely on \emph{a priori} assumption that  the {gradients} of the cost functions are uniformly bounded.
 
 In the recent work \cite{CKY}, the authors proved that for each $1\leq i \leq n$, the sequence $\{z_i (t)\}_{t \geq 0}$  of the gradient-push algorithm \eqref{eq-1-2} with constant stepsize $\alpha >0$ converges linearly to an $O(\alpha)$-neighborhood of the \mbox{minimizer $x_*$} when each cost $f_j$ is $L$-smooth and $f$ is $\beta$-strongly convex, if the stepsize satisfies $\alpha \leq \frac{c}{L^2}$ where the value $c>0$ is deterimined by $\beta$ and $W$. We refer to Theorem \ref{thm-2-1} below for the full detail. Having this result, one may wonder the full range of $\alpha >0$ for which this kind of convergence result holds for  the gradient-push algorithm \eqref{eq-1-2}. 
To discuss this issue, we introduce the following definition. 
\begin{defn} We denote by $\alpha_{GP} >0$   the largest constant such that for $\alpha \in (0, \alpha_{GP}]$ the sequence $\{z_i (t)\}_{t \geq 0}$ of the gradient-push algorithm converges linearly up to an $O(\alpha)$-neighborhood of the minimizer for each $1 \leq i \leq n$.
\end{defn} 
It is worth to mention that if the matrix $W$ is symmetric and doubly stochastic, then the gradient-push becomes the diffusion algorithm \cite{CS} which closely resembles the distributed gradient descent algorithm. The result \cite{YLY} showed that the DGD algorithm converges to an $O(\alpha)$-neighborhood of the minimizer linearly if each local cost is convex and $L$-smooth and the aggregate cost is strongly convex when the stepsize satisfies $\alpha \leq \frac{c}{L}$ with a specific value $c>0$ independent of $L>0$. A similar result was obtained in \cite{BBKW} for the diffusion algorithm. These results leads one to consider the following problem for the gradient-push algorithm.

\smallskip

\noindent \textbf{Open problem.} Is there a class of cost functions such that $\alpha_{GP} \geq \frac{c}{L}$ holds with a constant $c>0$ independent of $L>0$?

\smallskip


We mention that the proof of the result of \cite{YLY} heavily relies on a  reformulation of the DGD by a gradient descent method of a cost defined on the product space $(\mathbb{R}^{d})^n$. 
This approach is no longer applicable for the gradient-push algorithm, which makes the problem difficult.

In this work, we introduce an essential framework for analyzing the gradient-push algorithm. Consequently, we settle down the above problem for two popular classes of the cost functions. To state the result, we recall a well-known result (see \cite{XXK}) that under Assumptions \ref{ass-1} and \ref{ass-2}, the matrix W = $\{W_{ij}\}_{1 \leq i,j \leq n}$ has a  right eigenvector $\pi$ =$(\pi_1 \cdots \pi_n)^\top $ associated with the eigenvalue 1 satisfying $\sum_{j=1}^n \pi_j =1$ and
\begin{equation*}
\sum_{j=1}^n W_{kj}\pi_j = \pi_{k}\quad \forall~1\leq k \leq n. \label{eq-1-7}
\end{equation*} 
The following is the main result of this paper answering the above open problem.
\begin{thm}\label{thm-1-2}Suppose  one of the following conditions holds true: 
\begin{enumerate}
\item[Case 1:] Each cost $f_k$ is $\mu_k$-strongly convex and $L_k$-smooth for $1 \leq k \leq n$. The stepsize is assumed to satisfy $\alpha \in (0, \alpha_0]$ with  $\alpha_0 =\min_{1 \leq k \leq n} \frac{2n\pi_k}{L_k +\mu_k}$.  
\item[Case 2:] Each cost $f_j$ is given by convex quadratic form as $f_j(x) = \frac{1}{2}x^\top A_jx +B_jx$ for  a symmetric positive semi-definite matrix $A_j \geq$ 0 and $B_j  \in \mathbb{R}^{1\times d}$ and the aggregate cost $f$= $\frac{1}{n}\sum_{k=1}^n f_k$ is $\mu$-strongly convex. 
The stepsize is assumed to satisfy $\alpha \in (0,\alpha_0]$ with  $\alpha_0 = \min_{1 \leq k \leq n}\frac{2n\pi_k}{L_k+\epsilon}$ for any fixed $\ep>0$. Here the smoothness constant $L_k$ of cost $f_k$ is given by the largest eigenvalue of $A_k.$
\end{enumerate} 
Then for each $1\leq k \leq n$, the sequence $\{z_k (t)\}_{t \geq 0}$ of the gradient-push algorithm \eqref{eq-1-2} with stepsize  $\alpha \in (0,\alpha_0]$  converges linearly to an $O(\alpha)$-neighborhood of the \mbox{minimizer of $f$.} 
\end{thm}

Table \ref{tab-1} summarizes the results of this paper and the previous result for the gradient-push algorithm with constant stepsize.  
\begin{table}[ht] 
\centering
\begin{tabular}{|c|c|c|c| }\cline{1-4}
Ref. & {Condition on the cost functions}&   \makecell{Learning \\ rate}  &Example 
\\ 
\hline
&&&\\[-1em]
\cite{CKY2} &\makecell{$f_j$ is $L$-smooth (not necessarily convex) \\ $f$ is strongly convex }  & $\alpha \leq \frac{c}{L^2}$ & \makecell{$f_j (x) =x^\top  A_j x + b_j x,$ 
\\  $\sum_{j=1}^{n} A_j >0$} 
\\
&&&\\[-1em]
\hline
&&&\\[-1em]
\makecell{This work \\ {\footnotesize (Theorem \ref{thm-1-2})}} &   $f_j$ is strongly convex  and $L$-smooth   &   $\alpha \leq \frac{c}{L}$ & \makecell{$f_j (x) =g_c (x) + \ep \|x\|^2$, \\ $g_c$ is convex and 
\\
$L$-smooth  $\&$ $\ep>0$}
\\
&&&\\[-1em]
\hline
&&&\\[-1em]
\makecell{This work \\{\footnotesize (Theorem \ref{thm-1-2})}} & \makecell{$f_j$ is convex quadratic and $L$-smooth\\ $f$ is strongly convex}  &  $\alpha \leq \frac{c}{L}$& \makecell{$f_j (x) =x^\top  A_j x + b_j x$, \\ $\sum_{j=1}^n A_j > 0$, $A_j \geq 0$}
\\
\hline
\end{tabular}
\vspace{0.1cm}
\caption{ \footnotesize The results of guaranteeing the linear convergence for the gradient-push algorithm to an $O(\alpha)$-neighborhood of $x_*$. Here $c>0$ is a value which is independent of $L>0$ and determined by the matrix $W$ and the constant for the strongly convexity. The notation $M>0$ (resp., $M \geq 0$) for a matrix $M$ means that $M$ is positive definite (resp., semi-definite).}\label{tab-1}
\end{table}
To obtain the result, we rewrite the gradient-push algorithm as a succinct form in terms of the variable $w_i (t)$, and then notice that the right hand side of the algorithm can be written as an operator $T_{\alpha}:\mathbb{R}^{nd}\rightarrow \mathbb{R}^{nd}$ along with a small perturbation (refer to \eqref{eq-1-3} and \eqref{eq-1-21} below). We then establish the contraction property of $T_{\alpha}$, which guarantees a unique fixed point $w^\alpha$ of $T_{\alpha}$ (see Theorem \ref{thm-2-3}). Furthermore, we show that the fixed points are uniformly bounded with respect to $\alpha \in (0,\alpha_0]$ (see \mbox{Theorem \ref{thm-2-4}}). Given these results, we prove that the gradient-push algorithm converges linearly to the fixed point $w^\alpha$ of $T_{\alpha}$ (see Theorem \ref{thm-2-5}). Lastly, we obtain an estimate for the distance between $w^\alpha$  and $n\pi \otimes x_*$ in the following form $\|w^\alpha - n\pi \otimes x_*\|_{\pi \otimes 1_d} = O(\alpha)$ (see Theorem \ref{thm-2-6}). Combining these results will directly yield the result of Theorem \ref{thm-1-2}. These steps with corresponding theorems of this paper are summarized as a diagram in Figure \ref{fig-flow}.

\begin{figure}[htbp] 
\includegraphics[height=5.2cm, width=15cm]{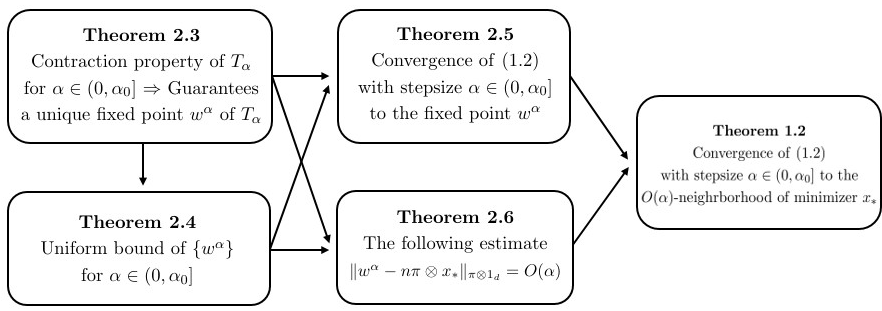} \caption{\footnotesize The flow of the main theorems for proving the convergence of the gradient-push algorithm \eqref{eq-1-2}}\label{fig-flow}
\end{figure}

\subsection{Combining the gradient-push with the Push-DIging}
Here we briefly discuss how the new understanding of the gradient-push algorithm with a large constant stepsize can be used effectively in combination with the Push-DIGing algorithm \cite{XXK, NOS}.  Determining a sharp lower bound of $\alpha_{GP}$ is a theoretically interesting problem. It is also important for practical scenarios since the choice of the stepsize $\alpha$ affects the performance of the algorithm. More precisely, choosing larger value $\alpha>0$ within the range $(0, \alpha_{GP}]$ usually makes the gradient-push algorithm converge faster to an $O(\alpha)$-neighborhood of the minimizer since the convergence rate is lower bounded by $(1-c\alpha)$ for a value $c>0$ independent of $\alpha \in (0, \alpha_{GP}]$. Over the convergence result up to an $O(\alpha)$-neighborhood, one can achieve the exact convergence algorithm to the minimizer by reducing the stepsize $\alpha$ repeatedly after suitable number of iterations. As an alternative way, we suggest to combine the gradient-push and the Push-DIGing algorithm \cite{XXK, NOS} which  converges to the minimizer exactly even for a constant stepsize by communicating also the gradient informations of the agents. The Push-DIGing algorithm is stated as follows:\
\begin{equation}\label{eq-1-20}
\begin{split}
 \mathbf{x}_i (t+1) & = \sum_{j=1}^n W_{ij} \mathbf{x}_j (t) - \alpha \mathbf{v}_i (t)
 \\
 \mathbf{y}_i (t+1)& =\sum_{j=1}^n W_{ij} \mathbf{y}_j (t)
 \\
 \mathbf{z}_i (t+1) & = \frac{\mathbf{x}_i (t+1)}{\mathbf{y}_i (t+1)}
 \\
 \mathbf{v}_i (t+1)& = \sum_{j=1}^n W_{ij} \mathbf{v}_j (t) + \nabla f_i (\mathbf{z}_i (t+1))- \nabla f_i (\mathbf{z}_i (t)),
 \end{split}
 \end{equation}
 where $\mathbf{v}_i (0) = \nabla f_i (\mathbf{z}_i (0))$. It was shown in the works \cite{XXK, NOS} that the Push-DIGing algorithm converges linearly to the minimizer when the constant  stepsize $\alpha >0$ is less than a specific value. Precisely, there are a value $M>0$  and a value $\bar{\alpha}>0$ such that for $\alpha \in (0,\bar{\alpha})$ we have
 \begin{equation*}
\sum_{k=1}^n \|z_k (t+1) -x_*\| \leq M (1-\gamma \alpha)^{t},
 \end{equation*}
where $\gamma >0$ is a value determined by the properties of the cost functions and the \mbox{matrix $W$}. Similarly to the gradient-push algorithm, we consider the following definition.
\begin{defn} We denote by $\alpha_{PD} >0$   the largest constant such that for $\alpha \in (0, \alpha_{PD}]$ the Push-DIGing algorithm converges linearly to the minimizer.
\end{defn} 
Although some lower bound of $\alpha_{PD}$ is achieved in \cite{XXK, NOS}, the exact value of $\alpha_{PD}>0$ has not been attained theoretically. We performed a numerical test to check that $\alpha_{PD} \ll \alpha_{GP}$ holds for a linear regression problem (refer to Figure \ref{fig1} for the detail). Consequently, the Push-DIGing is much slower than the gradient-push algorithm up to a certain number of iterations although the Push-DIGing converges to the minimizer exactly. Figure \ref{fig1} shows the convergence result with respect to the iteration $t\geq 0$, for both algorithms with a regularized regression problem given by
\begin{equation*}
\min_{x \in \mathbb{R}^d}\Big\{f(x):=\frac{1}{2n}\sum_{j=1}^n \Big(\| A_jx-B_j\|^2 + \delta \|x\|^2\Big)\Big\}
\end{equation*}
where $A_j \in \mathbb{R}^{m\times d}$ and $B_j \in \mathbb{R}^m$ are random matrices whose elements are selected uniformly randomly in the interval $[0,1]$ with dimensions $m=10$ and  $d=10$. We also set $\delta =0.1$ and the numbers of agents is $n=20$ with the mixing matrix $W$ generated as in Section \ref{sec-6}. We search a stepsize $\alpha_1 >0$ such that the performance of the Push-DIging algorithm \eqref{eq-1-20} is maximized. For this we test the algorithm \eqref{eq-1-20} with $\alpha =10^{-3} + 5\times10^{-6}k$ by increasing $k \in \mathbb{N}\cup \{0\}$. It is checked that the algorithm \eqref{eq-1-20} performs the best with $\alpha_1 =0.001175$ while the algorithm with $\alpha \in \{1.01 \alpha_1, 1.02 \alpha_1, 1.03\alpha_1\}$ diverge and the algorithm with $\alpha = 0.9\alpha_1$ converges slower than the case $\alpha =\alpha_1$. We also select a stepsize $\alpha_0 = 0.0297$ for the gradient-push algorithm  numerically for which the algorithm performs well. 
\begin{figure}[htbp]
\includegraphics[height=5cm, width=7.3cm]{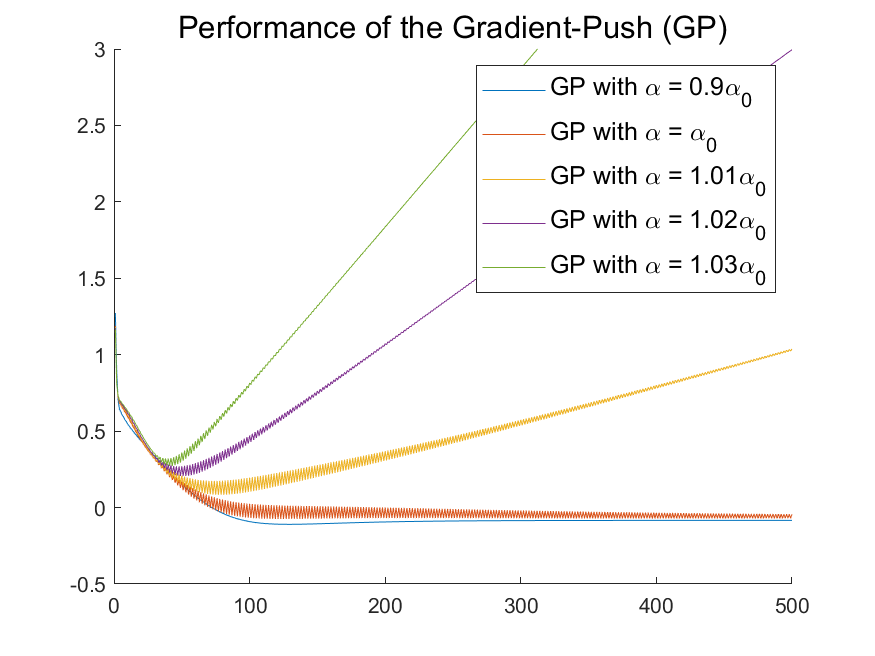}
\includegraphics[height=5cm, width=7.3cm]{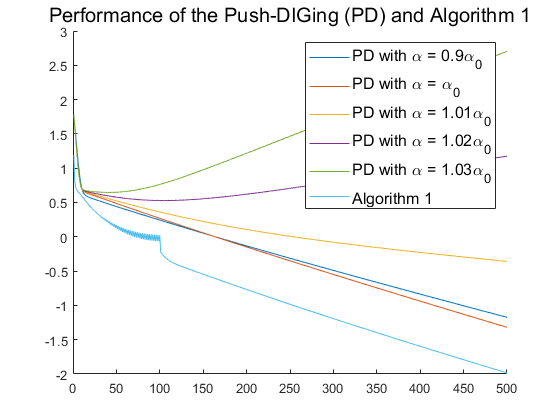}  
\vspace{-0.3cm}\caption{ Left : The graphs  of $\log_{10}\Big(\sum_{k=1}^n \|z_k (t)-x_*\|\Big)$ for the gradient-push ($\alpha_0 = 0.0297$). Right : The graphs of $\log_{10}\Big(\sum_{k=1}^n \|z_k (t)-x_*\|\Big)$ for the Push-DIGing ($\alpha_1 = 0.001175$). }\label{fig1}
\end{figure}

 To take advantages of the both algorithms,  we suggest a hybrid algorithm which performs the gradient-push algorithm for the first phase and then performs the Push-DIGing algorithm for the second phase. It is formally stated in Algorithm \ref{algo}.
 \begin{algorithm}
    \begin{algorithmic}[1] 
    \caption{Hybrid algorithm}\label{algo}
    \REQUIRE Initialize $x_i(0)$ arbitraily, $y_i(0)=1$ for all $i\in\{1,\cdots,m\}$.
            \FOR{$k =1 \ \textrm{to} \ GP_{iterate}$,}
                
                \STATE perform the gradient-push algorithm \eqref{eq-1-2} with $\alpha =\alpha_{0}$. 
                  
            \ENDFOR 
            
            \STATE  Set $\mathbf{x}_i (0) = w_i (GP_{iterate})$, $\mathbf{y}_i (0)=y_i (GP_{iterate})$, $\mathbf{z}_i (0)=z_i (GP_{iterate})$, $\mathbf{v}_i (0) =\nabla f(z_i(GP_{iterate}))$.
            
            \FOR{$k=GP_{iterate}+1 \ \textrm{to} \ Total_{iterate}$,}
                
                \STATE  perform the Push-DIGing algorithm \eqref{eq-1-20} with $\alpha =\alpha_1$.
                  
            \ENDFOR 
    \end{algorithmic}
    \end{algorithm}
Figure \ref{fig1} exhibits a graph of Algorithm \ref{algo} with the stepsizes $\alpha_0 = 0.0297$ and $\alpha_1= 0.001175$ companied with the iteration numbers $(GP_{iterate})=100$ and $(Total_{iterate})=500$. The graph of \mbox{Figure \ref{fig1}} shows the superior performance of the hybrid algorithm over the Push-DIGing algorithm, and illustates how the gradient-push algorithm with large stepsize may boost up the Push-DIging algorithm. Proving the convergence result of the hybrid algorithm will be addressed in a future work.

 This paper is organized as follows. In Section \ref{sec-2}, we state the main intermediate results related to the operator $T_{\alpha}$ for achieving the convergence result of Theorem \ref{thm-2-1}. Section \ref{sec-3} is devoted to show the contraction property of $T_{\alpha}$. Based on this result, we prove the convergence of the gradient-push algorithm towards the fixed point of $T_{\alpha}$ in Section \ref{sec-4}. Finally, we achieve a sharp estimate for the distance bewteen the fixed point of $T_{\alpha}$ and the minimizer of $f$ in Section \ref{sec-5}, which completes the convergence result  of the gradient-push algorithm. Section \ref{sec-6} provides numerical tests which support the theoretical results and the high-performance of the hybrid algorithm.

\medskip 

\noindent \textbf{Notation}
\begin{itemize}
\item We use the notation $1_d=(1,\ldots,1)^\top \in\mathbb{R}^d$.
\item We define the weighted $\pi$-norm by $\lVert w \rVert_\pi=(\sum_{j=1}^n\frac{1}{\pi_j}| w_j |^2)^{\frac{1}{2}}$ for $w=(w_1, \ldots, w_n)^\top \in\mathbb{R}^n$.
\item We use the weighted ($\pi\otimes 1_d$)-norm by $\lVert w \rVert_{\pi\otimes1_d}=(\sum_{j=1}^n\frac{1}{\pi_j}\lVert w_j \rVert^2)^{\frac{1}{2}}$ for $w=(w_1^\top ,\ldots,w_n^\top )^\top \in\mathbb{R}^{nd}$ where $\otimes$ is a Kronecker product.
\item We denote the induced matrix norms as $|||\cdot|||, |||\cdot|||_\pi, |||\cdot|||_{\pi\otimes1_d}$.
\item Denote by $\lambda(A )$ the set of all eigenvalues of $A$ for a symmetric matrix $A$.
\item We define $\nabla F: \mathbb{R}^{nd}\rightarrow \mathbb{R}^{nd}$ by $\nabla F(w) = (\nabla f_1 (w_1)^\top , \cdots, \nabla f_n (w_n)^\top )^\top $ for $w = (w_1^\top , \cdots, w_n^\top )^\top  \in \mathbb{R}^{nd}$ and use the notation
\begin{equation*}
\frac{x}{n\pi} := \Big( \frac{x_1^\top}{n\pi_1}, \cdots, \frac{x_n^\top}{n\pi_n}\Big)^\top
\end{equation*}
 for   $x=(x_1^\top ,\cdots,x_n^\top )^\top  \in \mathbb{R}^{nd}$. .
\item We use the notations $L = \max_{1 \leq k \leq n} L_k$ and $\bar{L} = \frac{1}{n}\sum_{k=1}^{n} L_k$. 
\end{itemize}   
\section{Convergence results}\label{sec-2}

In this section, we state preliminary  main results for showing the convergence result of Theorem \ref{thm-1-2}.
To recall the result of  \cite{CKY2}, we define the following constant
\begin{equation*}
\delta = \max_{t \in \mathbb{N}_0} \max_{1\leq i \leq n} \frac{1}{y_i (t)}
\end{equation*}
and  consider the following value 
\begin{equation}\rho := ||| W-W^{\infty} |||_\pi. 
\end{equation}
It was proved in \cite{XSKK} that we have $\rho \in (0,1)$ when Assumption 1 holds true. 
\begin{thm}[\cite{CKY2}]\label{thm-2-1} Suppose that each cost is $L$-smooth and the aggregate cost $f$ is $\beta$-strongly convex for some $\beta >0$. Let $\gamma = \frac{\beta L}{\beta + L}$, $q = \frac{n\gamma}{4L \delta}$  and assume that $\alpha>0$ satisfies  $\alpha \leq \frac{2}{L+\beta}$ and $\alpha < Q$ with $Q>0$ defined as
\begin{equation}\label{eq-4-20}
Q = \frac{q(1-\rho)}{L \rho (\delta q + \delta {\|1_n-n\pi\|_{\pi\otimes1_d}} + \sqrt{\sum_{j=1}^n \frac{1}{\pi_j}})}.
\end{equation}
 Then for each $1 \leq i \leq n$, the sequence $\{z_i (t)\}_{t \geq 0}$ converges to an $O(\alpha)$-neighborhood of the optimal point $x_*$ exponentially fast.
\end{thm}

\begin{rem}Note that $\gamma < \beta$ and $q < \frac{n\beta}{4L\delta}$. Using this we find
\begin{equation}\label{eq-2-10}
Q < \frac{q(1-\rho)}{L \rho (  \sqrt{\sum_{j=1}^n \frac{1}{\pi_j}})} < \frac{n\beta (1-\rho)}{4L^2 \rho \delta  \sqrt{\sum_{j=1}^n \frac{1}{\pi_j}}}.
\end{equation}
From this we find that $Q = O(1/L^2)$ and so the range $(0,Q)$ for the stepsize of the convergence result of Theorem \ref{thm-2-1} is quite restrictive when the smoothness constant $L>0$ is large.
\end{rem} 
The result of Theorem \ref{thm-1-2} improves the range of the stepsize given in Theorem \ref{thm-2-1} from $\alpha = O(1/L^2)$ to $\alpha = O(1/L)$ when each local cost is convex. To obtain this result, we begin by writing the gradient-push algorithm as follows:
\begin{equation}\label{eq-1-3}
w_i(t+1) = \sum_{j=1}^n W_{ij} \Big(w_j(t) - \alpha \nabla f_j\Big(\frac{w_j(t)}{y_j(t)}\Big)\Big),
\end{equation}
and recall from \cite{CKY} that there exist $a >0$  such that 
\begin{equation}\label{eq-1-6}
\Big\Vert \frac{1}{y_j(t)} - \frac{1}{n\pi_j} \Big\Vert \leq a \rho^t
\end{equation}
for all $1 \leq j \leq n$. Now we introduce the mapping  $T_{\alpha}$ : $\mathbb{R}^{nd}$ $\rightarrow$ $\mathbb{R}^{nd}$ defined as
\begin{equation}\label{eq-1-21}
T_{\alpha}(w) =\Big(
\sum_{j=1}^n W_{1j}\Big(w_j - \alpha\nabla f_j\Big(\frac{w_j}{n\pi_j}\Big)\Big)^\top , \cdots, 
\sum_{j=1}^n W_{nj}\Big(w_j - \alpha\nabla f_j\Big(\frac{w_j}{n\pi_j}\Big)\Big)^\top \Big)^\top 
\end{equation}
and write the algorithm \eqref{eq-1-3} in the following way
\begin{equation}\label{eq-1-4}
w(t+1) = T_{\alpha}\big(w(t)\big) + \alpha P_t(w(t)),
\end{equation}
where $P_t$ : $\mathbb{R}^{nd}$ $\rightarrow$ $\mathbb{R}^{nd}$ is defined as follows
\begin{equation}\label{eq-1-5}
P_t(w)=\Big(
\sum_{j=1}^n W_{1j}\Big(\nabla f_j \Big(\frac{w_j}{n\pi_j}\Big) - \nabla f_j \Big(\frac{w_j}{y_j(t)}\Big)\Big)^\top , \cdots, \sum_{j=1}^n W_{nj}\Big(\nabla f_j\Big(\frac{w_j}{n\pi_j}\Big) - \nabla f_j\Big(\frac{w_j}{y_j(t)}\Big)\Big)^\top \Big)^\top 
\end{equation}
for $w=(w_1, \cdots, w_n) \in \mathbb{R}^{nd}$.
Using the   convergence property \eqref{eq-1-6} of $y_j (t)$ towards $n\pi_j$, we show that the operator $P_t (w)$ converges to zero exponentially fast when $w$ is contained in a compact set (See Lemma \ref{lem-4-1}). Therefore, the convergence property of the algorithm \eqref{eq-1-4} is closely related to the property of the mapping $T_{\alpha}$. This motivates us to study the contraction property of the mapping $T_{\alpha}$ in the following theorem.
\begin{thm}\label{thm-2-3}\mbox{~}
\begin{enumerate} 
\item[Case 1:] Each cost $f_k$ is $\mu_k$-strongly convex and $L_k$-smooth for $1 \leq k \leq n$. The stepsize is assumed to satisfy $\alpha \in (0, \alpha_0]$ with  $\alpha_0 =\min_{1 \leq k \leq n} \frac{2n\pi_k}{L_k +\mu_k}$.  
\item[Case 2:] Each cost $f_k$ is given by convex quadratic form as $f_k(x) = \frac{1}{2}x^\top A_kx +B_kx$ for  a symmetric positive semi-definite matrix $A_k$ and $B_k \in \mathbb{R}^{1\times d}$ and the aggregate cost $f$= $\frac{1}{n}\sum_{k=1}^n f_k$ is $\mu$-strongly convex. 
The stepsize is assumed to satisfy $\alpha \in (0,\alpha_0]$ with  $\alpha_0 = \min_{1 \leq k \leq n}\frac{2n\pi_k}{L_k+\epsilon}$ for any fixed $\ep>0$. Here the smoothness constant $L_k$ of cost $f_k$ is given by $\max \{\lambda(A_k)\}$.
\end{enumerate} 
For the above two cases, the mapping $T_{\alpha}$ is contractive with respect to $\Vert . \Vert_{\pi\otimes1_d}$. More precisely, for $\alpha \in (0,\alpha_0]$ we have
\begin{equation}\label{eq-2-11}
 \lVert T_{\alpha}(x)- T_{\alpha}(y) \rVert_{\pi\otimes1_d} \leq (1-C\alpha )\lVert x-y \rVert_{\pi\otimes1_d}\quad \forall~ x,y \in \mathbb{R}^{nd},
\end{equation}
where $C>0$ is defined as
\begin{equation}\label{eq-2-12}
C= \Bigg\{ \begin{array}{ll} \min_{1\leq k \leq n}\frac{\mu_k L_k}{n(\mu_k+L_k)\pi_{k}}&\quad \textrm{for Case 1}
\medskip
\\
\frac{(1-\eta_{\ep})}{\alpha_0}& \quad \textrm{for Case 2},
\end{array}
\end{equation}
and $\eta_{\ep} >0$ for Case 2 denotes the contraction  constant of $T_{\alpha_0}$ which will be verified to be $\eta_{\ep} \in (0,1)$.
\end{thm}
 By the contraction property of the above result, the mapping $T_{\alpha}$ for $\alpha \in (0,\alpha_0]$ has a unique fixed point, which we denote by  $w^\alpha \in \mathbb{R}^{nd}$. 
 We show that the fixed points $w^\alpha$ are uniformly bounded for $\alpha \in (0, \alpha_0]$ in the following theorem.

\begin{thm}\label{thm-2-4}
For the two cases  of Theorem \ref{thm-2-3}, there exists a unique fixed point $w^\alpha$ of $T_\alpha$ for each $\alpha \in (0,\alpha_0]$. In addition, they are uniformly bounded, i.e.
\begin{equation*}
\sup_{\alpha \in (0,\alpha_0]} \Vert w^\alpha\Vert_{\pi\otimes1_d} \leq R
\end{equation*}
where $R=\frac{1}{C}\|J(\nabla F(0))\|_{\pi\otimes1_d}$ and $C$ is defined in Theorem \ref{thm-2-3}.
\end{thm}
The proofs of Theorem \ref{thm-2-3} and Theorem \ref{thm-2-4} will be given in Section \ref{sec-3}. 
The results obtained in the above two results on the mapping $T_{\alpha}$ allow us to study the convergence property of the gradient-push algorithm \eqref{eq-1-3}. 
 Namely, we establish a  sharp estimate of  $\Vert w(t) - w^\alpha \Vert_{\pi\otimes1_d}$.
\begin{thm}\label{thm-2-5}\mbox{~} Consider the two cases and the value $\alpha_0 >0$ of Theorem \ref{thm-2-3}. Fix an initial point $w(0) \in \mathbb{R}^{nd}$ and  denote by $w_{\alpha}(t)$ the sequence $w(t) \in \mathbb{R}^{nd}$ of the algorithm \eqref{eq-1-3} with stepsize $\alpha >0$ and initial point $w(0)$.  Then, for $\alpha \in (0,\alpha_0]$ the sequence $w_\alpha(t)$ converges to the unique fixed point $w^\alpha$ of $T_{\alpha}$ with linear convergence rate satisfying the following estimate
\begin{equation*} 
\Vert w_{\alpha}(t+1)-w^\alpha\Vert_{\pi\otimes1_d} \leq V_\alpha (1-C\alpha)^{t+1}\Vert w(0)-w^\alpha\Vert_{\pi\otimes 1_d} + \mathcal{R}(t),
\end{equation*} 
where
\begin{equation*}
\mathcal{R}(t) =\bigg\{ \begin{array}{ll}tV_\alpha\alpha bR\rho^t +\alpha bR\rho^t &\textrm{if}~  1-C\alpha=\rho
\\
 \frac{\alpha bRV_\alpha(1-C\alpha)}{1-C\alpha-\rho}\Big[(1-C\alpha)^t-\rho^t\Big] +\alpha bR\rho^t &\textrm{if}~ 1-C\alpha \neq \rho.
 \end{array}
\end{equation*}
Here $V_{\alpha}>0$ is defined as
\begin{equation}\label{eq-1-9}
V_\alpha = \prod_{j=0}^{\infty}\big(1+ \frac{\alpha b\rho^j}{1-C\alpha}\big),
\end{equation}
where we have set $b=aL$ with the constant $a>0$  defined in \eqref{eq-1-6}. We will show $V_\alpha < \infty$ in Section \ref{sec-4}.

\end{thm}
The proof of this result will be given in Section \ref{sec-4}. In the following result, we estimate the distance between the fixed point $w^\alpha$ of $T_{\alpha}$ and minimizer $x_*$ of the aggregate \mbox{cost $f$.}  
\begin{thm}\label{thm-2-6}
Consider the two cases and the value $\alpha_0 >0$ of Theorem \ref{thm-2-3}. Then, for $\alpha \in (0,\alpha_0]$ we have
\begin{equation*}
\Vert w^\alpha-n\pi \otimes x_* \Vert_{\pi\otimes 1_d}  \leq \frac{\alpha \rho}{1-\rho}\Big(1+\frac{L}{\gamma}\sqrt{\sum_{k=0}^n\frac{1}{\pi_k}}\Big)\Big( \frac{LR}{n\pi_{min}}+Q\Big),
\end{equation*}
where $\gamma=\frac{\bar{L}\mu}{\bar{L}+\mu}$, $Q=\|\nabla F(0)\|_{\pi\otimes 1_d}$ and $\pi_{min}=\min_{1\leq k\leq n}\pi_k$.
\end{thm}
The proof of this result is given in Section \ref{sec-5}.   Combining this fact with Theorem \ref{thm-2-5}, we will prove the result of Theorem \ref{thm-1-2} that  $w_\alpha(t)$ converges to the $O(\alpha)$-neighborhood   of minimizer $x_*$.

\section{Properties of $T_{\alpha}$}\label{sec-3}

In this section, we investigate the contraction property of the mapping $T_{\alpha}$ defined in \eqref{eq-1-21} and show that the fixed point $w^\alpha$ of $T_{\alpha}$ is uniformly bounded for $\alpha \in (0,\alpha_0]$. Namely, we   prove the results of Theorem \ref{thm-2-3} and Theorem \ref{thm-2-4}.
  
We begin with recalling the following Lemma. 
\begin{lem}\label{lem-3-1}Suppose g: $\mathbb{R}^{nd}$ $\rightarrow$ $\mathbb{R}$ is $\mu$-strongly convex and $L$-smooth for some $L$$>$0 and $\mu$$>$0. Then 
\begin{equation*}
\langle x-y, \nabla g(x) - \nabla g(y) \rangle \geq \frac{\mu L}{\mu + L}\Vert x-y \Vert^2 + \frac{1}{\mu + L} \Vert  \nabla g(x) - \nabla g(y) \Vert^2.
\end{equation*}
holds for all $x$,$y$ $\in$ $\mathbb{R}^{nd}$
\end{lem}
\begin{proof} We refer to \cite{B} for a proof of this Lemma. 
\end{proof}
Using this Lemma, one may prove the following well-known result.
\begin{lem}\label{lem-3-2}Suppose g: $\mathbb{R}^{d}$ $\rightarrow$ $\mathbb{R}$ is $\mu$-strongly convex and $L$-smooth for some $L  > 0$ and $\mu > 0$. Then for any $\alpha \in \big(  0,  \frac{2}{\mu + L} \big]$, the following estimate holds.
\begin{equation*}
\Vert x-y- \alpha(\nabla g(x) - \nabla g(y)) \Vert \leq (1 -\gamma\alpha)\Vert x-y \Vert \quad \forall~x, y \in \mathbb{R}^{d},
\end{equation*}
where  $\gamma =\frac{\mu L}{\mu + L}$. 
\end{lem}
\begin{proof}We give a proof for the reader's convenience. Using Lemma \ref{lem-3-1}, we get
\begin{equation*}
\begin{split}
&\Vert  x-y- \alpha(\nabla g(x) - \nabla g(y)) \Vert^2
\\
 &= \Vert x-y \Vert^2 -2\alpha\langle x-y, \nabla g(x) - \nabla g(y) \rangle + \alpha^2 \Vert \nabla g(x) - \nabla g(y)  \Vert^2
\\
&\leq \Big( 1- \frac{2\alpha \mu L}{\mu+L}\Big)\Vert x-y \Vert^2 +\Big( \alpha^2   -\frac{2\alpha}{\mu + L} \Big)\Vert  \nabla g(x) - \nabla g(y) \Vert^2.
\end{split}
\end{equation*}
This inequality for $\alpha \leq \frac{2}{\mu +L}$ yields 
\begin{equation*}
\begin{split}
\Vert  x-y- \alpha(\nabla g(x) - \nabla g(y)) \Vert^2 &\leq \Big(1 - \frac{2\mu L}{\mu + L}\alpha\Big) \Vert x-y \Vert^2
\\
&\leq \Big(1- \frac{\mu L}{\mu + L}\alpha\Big)^2 \Vert x-y \Vert^2.
\end{split}
\end{equation*}
The proof is done.
\end{proof}
 Based on this lemma, we derive the following result.
\begin{lem}\label{lem-3-30}
Suppose that for $1\leq k \leq n$, the local cost $f_k$ is $L_k$-smooth and $\mu_k$-strongly convex. Then, for ${\alpha} \leq \Big(0,~ \min_{1 \leq k \leq n}\frac{2{n\pi_k}}{\mu_k +L_k}\Big]$ we have
\begin{equation*}
\Big\| w- v- \alpha \Big( \nabla F\Big( \frac{w}{n\pi}\Big) - \nabla F\Big( \frac{v}{n\pi}\Big)\Big) \Big\|_{\pi\otimes1_d} \leq ( 1- C\alpha) \|w-v\|_{\pi\otimes1_d},
\end{equation*}
where $C=\min_{1 \leq k \leq n} \frac{\mu_k L_k }{n \pi_k (\mu_k +L_k)}$.
\end{lem}
\begin{proof}If the stepsize satisfies $\frac{\alpha}{n\pi_k} \leq \frac{2}{\mu_k +L_k}$ for  each $1\leq k \leq n$, we have
\begin{equation*}
\Big\| \frac{w_k}{n\pi_k} -\frac{v_k}{n\pi_k} -\frac{\alpha}{n\pi_k} \Big( \nabla f_k \Big( \frac{w_k}{n\pi_k}\Big) - \nabla f_k \Big( \frac{v_k}{n\pi_k}\Big)\Big)\Big\| \leq \Big(1 - \frac{\mu_k L_k}{\mu_k +L_k} \cdot \frac{\alpha}{n\pi_k}\Big) \Big\|\frac{w_k}{n\pi_k} -\frac{v_k}{n\pi_k}\Big\|,
\end{equation*}
which gives
\begin{equation*}
\Big\| w_k -v_k -\alpha \Big( \nabla f_k \Big( \frac{w_k}{n\pi_k}\Big) - \nabla f_k \Big( \frac{v_k}{n\pi_k}\Big)\Big)\Big\| \leq \Big(1 - \frac{\mu_k L_k}{\mu_k +L_k} \cdot \frac{\alpha}{n\pi_k}\Big) \|{w_k}-{v_k}\|.
\end{equation*}
This directly yields
\begin{equation*}
\begin{split}
\Big\| w- v- \alpha \Big( \nabla F\Big( \frac{w}{n\pi}\Big) - \nabla F\Big( \frac{v}{n\pi}\Big)\Big) \Big\|_{\pi\otimes1_d} &\leq \max_{1 \leq k \leq n} \Big( 1- \frac{\mu_k L_k \alpha}{n \pi_k (\mu_k +L_k)}\Big) \|w-v\|_{\pi\otimes1_d}
\\
&\leq  \Big( 1- \Big[\min_{1 \leq k \leq n}\frac{\mu_k L_k }{n \pi_k (\mu_k +L_k)}\Big]\alpha\Big) \|w-v\|_{\pi\otimes1_d}.
\end{split}
\end{equation*}
The proof is done.
\end{proof}
Define a function $J(w) :\mathbb{R}^{nd}\rightarrow \mathbb{R}^{nd}$ as
\begin{equation}\label{eq-3-18}
J(w) = \Big(
\sum_{j=1}^n W_{1j}w_j^\top , \cdots,\sum_{j=1}^n W_{nj}w_j^\top \Big)^\top \quad \textrm{for}~w=(w_1^\top, \cdots, w_n^\top)^\top \in \mathbb{R}^{nd}.
\end{equation}
We have the following result.
\begin{lem}\label{lem-3-3}
We have
\begin{equation}\label{eq-3-26}
\|J(x)\|_{\pi\otimes1_d} \leq  \|x\|_{\pi\otimes1_d}\quad \forall~x  \in \mathbb{R}^{nd}.
\end{equation}
In addition, the equality holds if and only if there exists a vector $\zeta \in \mathbb{R}^d$  such that
 \begin{equation}\label{eq-3-15}
x_j=\pi_j \zeta\quad \textrm{for}~1 \leq j \leq n.
\end{equation}
\end{lem}
\begin{proof}
By the triangle inequality, for any $x \in \mathbb{R}^{nd}$
\begin{equation}\label{eq-3-20}
\begin{split}
\|J(x)\|_{\pi\otimes1_d} &=\Big[\sum_{k=1}^n\frac{1}{\pi_k} \Vert \sum_{j=1}^n W_{kj}x_j\Vert^2\Big]^{\frac{1}{2}}
\\
&\leq \Big[\sum_{k=1}^n\frac{1}{\pi_k}  \Big(\sum_{j=1}^n W_{kj}\Vert x_j\Vert\Big)^2\Big]^{\frac{1}{2}}
\\
&=\Big[\sum_{k=1}^n\frac{1}{\pi_k} \Big( \sum_{j=1}^n \sqrt{W_{kj} \pi_j}\frac{\sqrt{W_{kj}}}{\sqrt{\pi_j}}\Vert x_j\Vert\Big)^2\Big]^{\frac{1}{2}}.
\end{split}
\end{equation}
We use the Cauchy-Schwartz inequality to get
\begin{equation}\label{eq-3-21}
\begin{split}
\|J(x)\|_{\pi\otimes1_d}
&\leq \Big[\sum_{k=1}^n\frac{1}{\pi_k}  \sum_{j=1}^n W_{kj}\pi_j\sum_{j=1}^n\frac{W_{kj}}{\pi_j}\Vert x_j\Vert^2\Big]^{\frac{1}{2}}
\\
&=\Big[\sum_{k=1}^n\sum_{j=1}^n\frac{W_{kj}}{\pi_j}\Vert x_j\Vert^2\Big]^{\frac{1}{2}}
\\
&=\Big[\sum_{j=1}^n\frac{1}{\pi_j}\Vert x_j\Vert^2\Big]^{\frac{1}{2}}=\|x\|_{\pi\otimes1_d},
\end{split}
\end{equation}
where we used \eqref{eq-1-7} in the first equality and  \eqref{eq-1-8} in the second equality.

Next we assume that $x \in \mathbb{R}^{nd}$ satisfies the equality of the inequality \eqref{eq-3-26}. Then the two inequalities in \eqref{eq-3-20} and \eqref{eq-3-21} should hold equality.  For each $1\leq k \leq n$, we set $I_k = \{ j \in \{1,2,3 \cdots , n\}  ~|~  W_{kj} \neq 0 \}$. 
By the equality condition of the Cauchy-Schwarz inequality of \eqref{eq-3-21}, the following ratios 
\begin{equation*}
\frac{\sqrt{W_{kj}}}{\sqrt{\pi_j}}\|x_j\|\Big/\sqrt{W_{kj}\pi_j}=\frac{\|x_j\|}{\pi_j} 
\end{equation*} are same  for all $j \in I_k$. Combining this with the fact that the graph induced by $\{W_{kj}\}$ is connected, we see that 
\begin{equation}\label{eq-3-25}
\frac{\|x_1\|}{\pi_1} =  \frac{\|x_2\|}{\pi_2}  = \cdots = \frac{\|x_n\|}{\pi_n}.
\end{equation}
As the inequality \eqref{eq-3-20} holds equality, we find that equalities hold in the following inequalities.
\begin{equation}\label{eq-1-13}
\Vert \sum_{j=1}^n W_{kj}x_j\Vert \leq \sum_{j=1}^n W_{kj}\Vert x_j \Vert \quad \forall 1 \leq k \leq n.
\end{equation}
Using this and the equality condition of the triangle inequality, for each $1\leq k \leq n$ there are nonnegative constants $c_{jk}$ for $j \in I_k$ such that $x_j = c_{jk} x_k$. Combining this with the fact that the network  induced by $\{W_{kj}\}$ is connected, we find that for any $1\leq j, k\leq n$, there are values $c_{jk} \geq 0$ such that $x_j = c_{jk}x_k$.  This, together with  \eqref{eq-3-25}, yields that there exists a vector $\zeta \in \mathbb{R}^d$ such that $x_j = \pi_j \zeta$ for $1\leq j \leq n$. It finishes the proof.
\end{proof}

 Now we give the proof of Theorem \ref{thm-2-3} for Case 1.
\begin{proof}[Proof of Theorem \ref{thm-2-3} for Case 1] 
Using the definition \eqref{eq-1-21} of $T_{\alpha}$ we find 
\begin{equation*}
\begin{split}
&\Vert T_\alpha(w) - T_\alpha(v)\Vert_{\pi\otimes1_d} 
\\
&= \Big(\sum_{k=1}^n\frac{1}{\pi_k}\Big\Vert \sum_{j=1}^n W_{kj}\Big(w_j-v_j-\alpha\Big(\nabla f_j\Big(\frac{w_j}{n\pi_j}\Big)-\nabla f_j\Big(\frac{v_j}{n\pi_j}\Big)\Big)\Big)\Big\Vert^2\Big)^{\frac{1}{2}}
\\
&=\Big\| J\Big(w-v-\alpha\Big(\nabla F\Big(\frac{w}{n\pi}\Big)-\nabla F\Big(\frac{v}{n\pi}\Big)\Big)\Big)\Big\|_{\pi\otimes1_d},
\end{split}
\end{equation*}
where we have set $\nabla F\Big(\frac{w}{n\pi}\Big) =\Big(\nabla f_1\Big(\frac{w_1}{n\pi_1}\Big)^\top , \cdots, \nabla f_n\Big(\frac{w_n}{n\pi_n}\Big)^\top \Big)^\top $. Now we apply Lemma \ref{lem-3-3} and Lemma \ref{lem-3-30}  to derive the following estimate 
\begin{equation*}
\begin{split}
\Vert T_\alpha(w) - T_\alpha(v)\Vert_{\pi\otimes1_d}
&\leq \| w-v-\alpha\Big(\nabla F\Big(\frac{w}{n\pi}\Big)-\nabla F\Big(\frac{v}{n\pi}\Big)\Big)\|_{\pi\otimes1_d}
\\
&\leq ( 1-C\alpha)\Vert w - v \Vert_{\pi\otimes1_d},
\end{split}
\end{equation*}
The proof is done.
\end{proof} 
In order to prove Theorem \ref{thm-2-3} for Case 2, we prepare the following Lemma.
\begin{lem}\label{lem-3-4}
Let $A$ $\in$ $\mathbb{R}^{d\times d}$ be a symmetric and positive semidefinite matrix. Then, for  any $\alpha$ $\in$ $(0, \frac{2}{\max \lambda(A)})$, the following inequality holds 
\begin{equation*}
\Vert \big(I_d -\alpha A\big)x \Vert \leq \Vert x \Vert. \quad \forall x \in \mathbb{R}^{d}.
\end{equation*}
The equality in the above inequality holds if and only if
\begin{equation*}
 Ax = 0. 
\end{equation*}
\end{lem}
\begin{proof}
Since $A$ is symmetric, it is diagonalizable in the sense that
\begin{equation*}
A = UDU^\top ,
\end{equation*}
where $D \in \mathbb{R}^{d\times d} = \textrm{diag}(\lambda_1, \cdots, \lambda_d)$ is a diagonoal matrix whose entries are eigenvalues of $A$ and $U \in \mathbb{R}^{d \times d}$ satisfies $UU^\top =U^\top U=I_d$. Using this decomposition,  we find the following equality
\begin{equation*}
\begin{split}
\Vert\big( I_d - \alpha A \big)x \Vert &=\Vert \big(I_d -\alpha UDU^\top \big)x \Vert
\\ 
&=\Vert \big(UU^\top -\alpha UDU^\top \big)x\Vert
\\
&=\Vert \big(U^\top -\alpha DU^\top \big)x\Vert
\\
&=\Vert \big(I_d -\alpha D\big)U^\top x \Vert.
\end{split}
\end{equation*}
Letting $w=U^\top  x\in \mathbb{R}^d$ in the above equality gives
\begin{equation}\label{eq-3-10}
\begin{split}
\Vert \big( I_d - \alpha A \big)x \Vert &=\Vert \big(I_d -\alpha D\big)w \Vert
\\
&=\Big(\sum_{k=1}^d (1-\alpha \lambda_k)^2w_k^2\Big)^{\frac{1}{2}}.
\end{split}
\end{equation}
Since $A$ is positive semi-definite, we have $\lambda_j \geq 0$ for $1 \leq j \leq d$. Combining this with the assumption $\alpha \in (0, \frac{2}{\max \lambda (A)})$, it follows that $-1 < 1 - \alpha \lambda_k \leq 1$. Therefore we have
\begin{equation}\label{eq-3-24} 
(1-\alpha \lambda_k)^2 w_k^2 \leq w_k^2 \quad \end{equation}
for any $1 \leq k \leq d$. Inserting this inequality in \eqref{eq-3-10} we get
\begin{equation}\label{eq-3-36}
\begin{split}
\Vert \big( I_d - \alpha A \big)x \Vert&\leq \Big(\sum_{k=1}^d w_k^2\Big)^{\frac{1}{2}} = \Vert w \Vert
=\Vert x \Vert,
\end{split}
\end{equation}
which is the desired inequality.

Now we proceed to prove the equality condition of the inequality. If $Ax=0$, then it is trivial that $\Vert \big(I_d -\alpha A\big)x \Vert = \Vert x \Vert$.

Conversely, we suppose that $\Vert \big(I_d -\alpha A\big)x \Vert = \Vert x \Vert$. Then for all $1\leq k \leq d$, the inequality \eqref{eq-3-24} should hold equality, i.e., 
\begin{equation*}
  (1-\alpha \lambda_k)^2w_k^2  =  w_k^2,
\end{equation*}
which leads to
\begin{equation*}
\alpha   (2-\alpha \lambda_k) \lambda_k w_k^2= 0.
\end{equation*}
Combining this with the fact that  $(2-\alpha \lambda_k) >0$ for $\alpha \in \Big( 0, \frac{2}{\max \lambda (A)}\Big)$,   it follows that
\begin{equation*}
\lambda_k w_k=0
\end{equation*}
for all $1 \leq k \leq d$. This implies that $Dw =0$, and so we have $ Ax=UDU^\top  x=UDw=0.$
The proof is done.

\end{proof}

 Using the above lemma, we proceed to prove Theorem \ref{thm-2-3} for Case 2.
\begin{proof}[Proof of Theorem \ref{thm-2-3} for Case 2] For each $\alpha \in \Big( 0, \min_{1 \leq k \leq n} \frac{2n \pi_k}{L_k}\Big)$, we aim to show that
\begin{equation*}
\sup_{w\neq v}\frac{\Vert T_{\alpha}(w)-T_{\alpha}(v)\Vert_{\pi\otimes1_d}}{\Vert w-v \Vert_{\pi\otimes1_d}} < 1\label{eq-1-11}.
\end{equation*}
 We define $L_{\alpha}$ $:$ $\mathbb{R}^{nd}$ $\rightarrow$ $\mathbb{R}^{nd}$ as $L_{\alpha}(x) =\Big(L_\alpha (x)_1^\top , \cdots, L_{\alpha}(x)_n^\top \Big)^\top $ with 
\begin{equation*}
\big(L_{\alpha}(x)\big)_k = \sum_{j=1}^n W_{kj}\Big(I_d - \frac{\alpha A_j}{n\pi_j}\Big)x_j \in \mathbb{R}^d
\end{equation*}
for $1 \leq k \leq n$. Using $\nabla f_j(x) = A_j x + B_j^\top$   we find 
\begin{equation*}
\Big(T_{\alpha}(w) - T_{\alpha}(v)\Big) = L_{\alpha}(w-v).
\end{equation*}
Therefore, the desired inequality \eqref{eq-1-11} is equivalent to 
\begin{equation}\label{eq-3-16}
\sup_{x\neq0}\frac{\|L_{\alpha}(x)\Vert_{\pi\otimes1_d}}{\Vert x\Vert_{\pi\otimes1_d}} < 1.
\end{equation} 
In order to show this inequality, we first show that the term in the left hand side is less or equal to $1$.

For any $w =(w_1^\top, \cdots, w_n^\top)^\top \in \mathbb{R}^{nd}$, we notice that
\begin{equation}\label{eq-3-13}
\begin{split}
\Vert L_{\alpha}(w)\Vert_{\pi\otimes1_d} &= \Big[\sum_{k=1}^n\frac{1}{\pi_k} \Big\Vert \sum_{j=1}^n W_{kj}\Big(w_j-\alpha A_j\frac{w_j}{n\pi_j} \Big)\Big\Vert^2\Big]^{\frac{1}{2}}
\\
&=\| J\Big(w-\alpha\Big(A\Big(\frac{w}{n\pi}\Big)\Big)\Big)\|_{\pi\otimes1_d},
\end{split}
\end{equation}
where we have let $A\Big(\frac{w}{n\pi}\Big) =\Big(\Big(A_1\frac{w_1}{n\pi_1}\Big)^\top , \cdots, \Big(A_n\frac{w_n}{n\pi_n}\Big)^\top \Big)^\top $.

Using  Lemma \ref{lem-3-3} and Lemma \ref{lem-3-4} with the fact that $\alpha \in (0, \frac{2n\pi_j}{\max \lambda (A_j)}]$ for all $1\leq j \leq n$, we have
\begin{align}
\Vert L_{\alpha}(w)\Vert_{\pi\otimes1_d}  &\leq \| w-\alpha\Big(A\Big(\frac{w}{n\pi}\Big)\Big)\|_{\pi\otimes1_d}\label{eq-3-14}
\\
&\leq \| w\|_{\pi\otimes1_d}.\label{eq-3-17}
\end{align}
Therefore we have
\begin{equation}\label{eq-3-11}
\sup_{w\neq 0}\frac{\Vert L_{\alpha}(w)\Vert_{\pi\otimes1_d}}{\Vert w \Vert_{\pi\otimes1_d}} \leq 1.
\end{equation}
Now we aim to show that the above inequality is strict. For this we note that, since $L_{\alpha}$ is linear,  
\begin{equation*}
\sup_{x\neq0}\frac{\Vert L_{\alpha}(x)\Vert_{\pi\otimes1_d}}{\Vert x \Vert_{\pi\otimes1_d}}
=\sup_{x \in \mathbb{S}^{nd-1}}\frac{\Vert L_{\alpha}(x)\Vert_{\pi\otimes1_d}}{\Vert x \Vert_{\pi\otimes1_d}},
\end{equation*}
where $\mathbb{S}^{nd-1} =\{ x \in \mathbb{R}^{nd} : \Vert x \Vert =1 \}$.

In order to show that the inequality \eqref{eq-3-11} is strict, we assume the contrary that there exists a vector $x =(x_1^\top, \cdots, x_n^\top)^\top \in  \mathbb{S}^{nd-1}$ such that
\begin{equation*}
\frac{\Vert L_{\alpha}(x) \Vert_{\pi\otimes1_d}}{\Vert x \Vert_{\pi\otimes1_d}}=1\label{eq-1-12}.
\end{equation*}
Then the inequalities of \eqref{eq-3-14} and \eqref{eq-3-17} with $w=x$ should hold equality.  As for the equality condition of \eqref{eq-3-17}, it follows from  Lemma  \ref{lem-3-4} the following identities
\begin{equation}\label{eq-1-14}
A_jx_j=0  \quad \forall 1 \leq j \leq n.
\end{equation}
Then the equality condition of \eqref{eq-3-14} is written as
\begin{equation*}
\|J(x)\|_{\pi \otimes 1_d} =\|x\|_{\pi \otimes 1_d}.
\end{equation*}
By the equality condition of Lemma \ref{lem-3-3}, there exists a vector $\zeta \in \mathbb{R}^d$ such that for all $1\leq j \leq n$,
\begin{equation}\label{eq-3-15}
x_j=\pi_j \zeta.
\end{equation}
Inserting this form into \eqref{eq-1-14}, we find $A_j(\pi_j \zeta) = 0$, and so $A_j \zeta = 0$. Summing this over $1 \leq j \leq n$, we get
\begin{equation}\label{eq-3-41}
\big(\sum_{j=1}^n A_j\big)\zeta=0. 
\end{equation}
Since the aggregate cost $f$ is assumed to be $\mu$-strongly convex, the matrix $\big(\sum_{j=1}^n A_j\big)$ is positive definite.
Thus, the equality \eqref{eq-3-41} implies that $\zeta=0$. From this and \eqref{eq-3-15} we have $x=0$.  This is a contradiction to the prior assumption that $x \in \mathbb{S}^{nd-1}$. Therefore the equality \eqref{eq-1-12} cannot hold and so \eqref{eq-3-16} holds true, and \eqref{eq-1-11} is verified. 
 
  Let $\alpha_0 = \min_{1 \leq j \leq n}\Big(\frac{2n\pi_j}{\{\max \lambda(A_j)\}+\ep}\Big)$ for a fixed value $\ep>0$. By \eqref{eq-1-11} there exists a value  $\eta_{\ep} \in (0,1)$ such that
\begin{equation*}
\Vert T_{\alpha_0}(w) -T_{\alpha_0}(v)\Vert_{\pi\otimes1_d} \leq \eta_{\ep}\Vert w-v \Vert_{\pi\otimes1_d}.
\end{equation*} 
For any $\alpha \in (0, \alpha_0]$, we find from \eqref{eq-1-21} and \eqref{eq-3-18} the following identity
\begin{equation*}
T_{\alpha}(w) = \frac{\alpha}{\alpha_0}T_{\alpha_0}(w) + \Big(1-\frac{\alpha}{\alpha_0}\Big)J(w).
\end{equation*}
By the triangle inequality and Lemma \ref{lem-3-3}, we obtain an upper bound of $\Vert T_\alpha(w) -T_\alpha(v)\Vert_{\pi\otimes1_d}$ as follows
\begin{equation*}
\begin{split}
\Vert T_{\alpha}(w) -T_{\alpha}(v) \Vert_{\pi\otimes1_d} &\leq \frac{\alpha}{\alpha_0}\Vert T_{\alpha_0}(w) - T_{\alpha_0}(v) \Vert_{\pi\otimes1_d} +\Big( 1 -\frac{\alpha}{\alpha_0}\Big) \Vert J(w-v) \Vert_{\pi\otimes1_d}
\\
&\leq \frac{\alpha}{\alpha_0}\Vert T_{\alpha_0}(w) - T_{\alpha_0}(v) \Vert_{\pi\otimes1_d} +\Big( 1 -\frac{\alpha}{\alpha_0}\Big) \Vert w-v \Vert_{\pi\otimes1_d}
\\
&\leq \frac{\alpha}{\alpha_0}\eta_{\ep}\Vert w-v \Vert_{\pi\otimes1_d} +\Big( 1 -\frac{\alpha}{\alpha_0}\Big) \Vert w-v \Vert_{\pi\otimes1_d}
\\
&= \Big(1-\frac{1-\eta_{\ep}}{\alpha_0}\alpha\Big) \Vert w-v \Vert_{\pi\otimes1_d}.
\end{split}
\end{equation*}
 The proof is done.
\end{proof}

\begin{proof}[Proof of Theorem \ref{thm-2-4}]
By Theorem \ref{thm-2-3}, the mapping $T_\alpha$ is  contractive for any $\alpha \in (0,\alpha_0]$, and so there exist a unique fixed point $w^\alpha \in \mathbb{R}^{nd}$ of the mapping $T_\alpha$. 

In order to prove the uniform bound of $\{w^{\alpha}\}_{\alpha \in (0,\alpha_0]}$, we construct a sequence converging to $w^{\alpha}$. Namely, for a fixed value $\alpha \in (0,\alpha_0]$, we consider the sequence $\{ u(t)\}_{t \geq 0} \subset \mathbb{R}^{nd}$ such that $u (0) = 0 \in \mathbb{R}^{nd}$ and $u (t+1) = T_{\alpha} (u(t))$ for $t \geq 0$.
   We use the contraction property \eqref{eq-2-11} of $T_{\alpha}$ to find
\begin{equation}\label{eq-3-40}
\begin{split}
\Vert u(t+1) - u(t)\Vert_{\pi\otimes1_d} & = \|T_{\alpha}^t (u(1)) - T_{\alpha}^{t} (u(0))\|_{\pi\otimes1_d}
\\
&\leq (1-C\alpha)^t\Vert u(1) - u(0)\Vert_{\pi\otimes1_d}.
\end{split}
\end{equation}
Next we use $w^\alpha = T_{\alpha}(w^\alpha)$ and  \eqref{eq-2-11} to deduce
\begin{equation*}
\begin{split}
\|u(t+1) -w^\alpha\|_{\pi\otimes1_d}  & = \|T_{\alpha} (u(t)) - T_{\alpha} (w^\alpha)\|_{\pi\otimes1_d}
\\
&\leq (1-C\alpha)\|u(t) -w^\alpha\|_{\pi\otimes1_d},
\end{split}
\end{equation*}
and so $u(t)$ converges to $w^\alpha$ as $t \rightarrow \infty$. Combining this with \eqref{eq-3-40}, we get
\begin{equation*}
\begin{split}
\Vert w^\alpha-u(0)\Vert_{\pi\otimes1_d} & \leq \sum_{t=0}^\infty \Vert u(t+1) - u(t)\Vert_{\pi\otimes1_d}
\\
&\leq \sum_{t=0}^\infty (1-C\alpha)^t \Vert u(1) - u(0)\Vert_{\pi\otimes1_d}
\\
&= \frac{1}{C\alpha}\Vert u(1) - u(0)\Vert_{\pi\otimes1_d}.
\end{split}
\end{equation*}
Now we use $u(0)=0\in \mathbb{R}^{nd}$ and $u(1)=\alpha  J(\nabla F(0))$ to find
\begin{equation*}
\Vert w^\alpha\Vert_{\pi\otimes1_d} \leq   \frac{1}{C}\Vert J(\nabla F(0))\Vert_{\pi\otimes1_d}.
\end{equation*}
Therefore the set of the fixed points $\{w^\alpha\}_{\alpha \in (0,\alpha_0]}$ is bounded. 
The proof is done.
\end{proof} 

\section{Convergence of the gradient-push algorithm to the fixed point   of $T_{\alpha}$}\label{sec-4}
In this section, we give the proof  of Theorem \ref{thm-2-5} regarding the convergence property of the gradient-push algorithm \eqref{eq-1-4} towards the fixed point $w^\alpha$ of the mapping $T_{\alpha}$. For this, we make use of the contraction property of the mapping $T_{\alpha}$ obtained in Theorem \ref{thm-2-3} along with a  bound of mapping $P_{t}(w)$ of \eqref{eq-1-5} attained in the following lemma.
\begin{lem}\label{lem-4-1}
The mapping $P_{t}(w)$ defined in \eqref{eq-1-5} satisfies the following inequality 
\begin{equation*}
\Vert P_t(w) \Vert_{\pi\otimes 1_d} \leq b\rho^t \Vert w \Vert_{\pi\otimes 1_d} \quad \forall~w \in \mathbb{R}^{nd} \quad \textrm{and}\quad t \geq 0,
\end{equation*}
where we have set $b= aL$ with constant   $a>0$  defined in \eqref{eq-1-6}.
\end{lem}
\begin{proof}
Using Lemma \ref{lem-3-3} and the $L$-smooth property of $f_j$, we deduce
\begin{equation*}
\begin{split}
\Vert P_t(w) \Vert_{\pi\otimes 1_d} ^2 &= \sum_{k=1}^n \frac{1}{\pi_k}\Big\Vert \sum_{j=1}^n W_{kj}\Big(\nabla f_j\Big(\frac{w_j}{n\pi_j}\Big) - \nabla f_j\Big(\frac{w_j}{y_j(t)}\Big)\Big) \Big\Vert^2 
\\
&\leq  \sum_{k=1}^n \frac{1}{\pi_k} \Big\Vert \nabla f_k\Big(\frac{w_k}{n\pi_k}\Big) - \nabla f_k\Big(\frac{w_k}{y_k(t)}\Big) \Big\Vert^2 
\\
&\leq L^2 \sum_{k=1}^n \frac{1}{\pi_k}\Big(\frac{1}{y_k(t)} -\frac{1}{n\pi_k} \Big)^2 \big\Vert w_k \big\Vert^2 
\\
&\leq a^2 L^2 \rho^{2t} \sum_{k=1}^n\frac{1}{\pi_k}\big\Vert w_k \big\Vert^2 = b^2\rho^{2t} \Vert w \Vert_{\pi\otimes 1_d}^2.  
\end{split}
\end{equation*}
The proof is done.
\end{proof}

In the following lemma, we give a bound of $V_{\alpha}$ defined in \eqref{eq-1-9}.
\begin{lem} \label{lem-4-2}
For $\alpha \in (0,\alpha_0]$, we have
\begin{equation*}
V_{\alpha} \leq \exp\Big(\frac{\alpha_0b}{1-C\alpha_0}\frac{1}{1-\rho}\Big).
\end{equation*}
\end{lem}
\begin{proof}
 For any $\alpha \in (0,\alpha_0]$, the following inequality holds
\begin{equation*}  
V_{\alpha}  = \prod_{k=0}^{\infty}\Big(1 +\frac{\alpha b \rho^k}{1-C\alpha}\Big)  \leq \prod_{k=0}^{\infty}\Big(1 +\frac{\alpha_0 b \rho^k}{1-C\alpha_0}\Big). 
\end{equation*}
Taking logarithm and using the fact that $\log(1+x) \leq x$, we get
\begin{equation*}
\begin{split}
\log \Big[\prod_{k=0}^{\infty}\Big(1 +\frac{\alpha_0 b \rho^k}{1-C\alpha_0}\Big)\Big]&= \sum_{k=0}^{\infty} \log\Big(1 + \frac{\alpha_0b\rho^k}{1-C\alpha_0}\Big) 
\\
&\leq \sum_{k=0}^{\infty} \frac{\alpha_0b}{1-C\alpha_0}\rho^k
\\
&=\frac{\alpha_0b}{(1-C\alpha_0)(1-\rho)}.
\end{split}
\end{equation*}
The proof is done.
\end{proof}
 Now we are ready to prove   Theorem \ref{thm-2-5}.
\begin{proof}[Proof of Theorem \ref{thm-2-5}]
Using the triangle inequality and \eqref{eq-1-4} with the fact that $w^\alpha$ is the fixed point of the mapping $T_{\alpha}$, we get
\begin{equation}\label{eq-4-1}
\begin{split}
\Vert w(t+1) - w^\alpha \Vert_{\pi\otimes 1_d} &= \Vert T_{\alpha}(w(t)) - T_{\alpha}(w^\alpha) + \alpha P(w(t)) \Vert_{\pi\otimes 1_d}
\\
&\leq \Vert T_{\alpha}(w(t)) - T_{\alpha}(w^\alpha) \Vert_{\pi\otimes 1_d} + \alpha \Vert P(w(t))  \Vert_{\pi\otimes 1_d}.
\end{split}
\end{equation}
Applying the contraction property of $T_\alpha$ established in  Theorem \ref{thm-2-3} and using \mbox{Lemma \ref{lem-4-1}}, we obtain
\begin{equation}\label{eq-4-2}
\begin{split}
\Vert w(t+1)-w^\alpha\Vert_{\pi\otimes 1_d} &\leq (1-C\alpha)\Vert w(t)-w^\alpha\Vert_{\pi\otimes 1_d} +\alpha b\rho^t\Vert w(t)\Vert_{\pi\otimes 1_d}
\\
&= (1-C\alpha)\Vert w(t)-w^\alpha\Vert_{\pi\otimes 1_d} +\alpha b\rho^t\Vert w(t)-w^\alpha+w^\alpha\Vert_{\pi\otimes 1_d}.
\end{split}
\end{equation}
We use the triangle inequality and the bound of $\|w^\alpha\|_{\pi \otimes 1_d}$ obtained in Theorem \ref{thm-2-4} to find
\begin{equation*}
\begin{split}
\Vert w(t+1)-w^\alpha\Vert_{\pi\otimes 1_d}&\leq (1-C\alpha +\alpha b\rho^t)\Vert w(t)-w^\alpha\Vert_{\pi\otimes 1_d} +\alpha b\rho^t\Vert w^\alpha\Vert_{\pi\otimes 1_d} 
\\
&\leq (1-C\alpha +\alpha b\rho^t)\Vert w(t)-w^\alpha\Vert_{\pi\otimes 1_d} +\alpha bR\rho^t.
\end{split}
\end{equation*}
Using the above inequality recursively, we get
\begin{equation} \label{eq-4-3}
\begin{split}
\Vert w(t+1)-w^\alpha\Vert_{\pi\otimes 1_d} &\leq \big[\underbrace{\prod_{i=0}^t(1-C\alpha +\alpha b\rho^i)}_{T_1}\big]\Vert w(0)-w^\alpha\Vert_{\pi\otimes 1_d} 
\\
& \quad+\alpha bR\underbrace{\sum_{i=0}^{t-1}\big[\rho^i \prod_{k=i+1}^t(1-C\alpha +\alpha b\rho^k)}_{T_2}\big] +\alpha bR\rho^t.
\end{split}
\end{equation}
In view of the definition \eqref{eq-1-9} of $V_{\alpha}$, we have 
\begin{equation}\label{eq-4-9}
T_1 = (1-C\alpha)^{t+1}\prod_{i=0}^t\Big(1+ \frac{\alpha b\rho^i}{1-C\alpha}\Big)\leq(1-C\alpha)^{t+1}V_{\alpha}.
\end{equation}
Next we estimate $T_2$. Notice that
\begin{equation*} 
\begin{split} 
\prod_{k=i+1}^t(1-C\alpha +\alpha b\rho^k) &=(1-C\alpha)^{t-i} \prod_{k=i+1}^{t}\Big(1+ \frac{\alpha b\rho^k}{1-C\alpha}\Big)
\\
&\leq(1-C\alpha)^{t-i} \prod_{k=0}^{\infty}\Big(1+ \frac{\alpha b\rho^k}{1-C\alpha}\Big)
\\
& = V_\alpha(1-C\alpha)^{t-i}.
\end{split}
\end{equation*} 
Using this inequality, we estimate $T_2$ as follows:
\begin{equation}\label{eq-4-10} 
\begin{split}
T_2 &\leq V_\alpha\sum_{i=0}^{t-1}\rho^i (1-C\alpha)^{t-i}
=\left\{\begin{array}{ll}  tV_\alpha \rho^t&~\textrm{if}~1-C\alpha = \rho
\\
\frac{V_\alpha(1-C\alpha)}{1-C\alpha-\rho}\Big[(1-C\alpha)^t-\rho^t\Big]&~\textrm{if}~1-C\alpha \neq \rho.
\end{array}
\right.
\end{split}
\end{equation} 
Putting \eqref{eq-4-9} and \eqref{eq-4-10} in  \eqref{eq-4-3}, we get
\begin{equation*}
\Vert w(t+1)-w^\alpha\Vert_{\pi\otimes 1_d} \leq V_\alpha (1-C\alpha)^{t+1}\Vert w(0)-w^\alpha\Vert_{\pi\otimes 1_d} +\mathcal{R}(t),
\end{equation*} 
where 
\begin{equation*}
\mathcal{R}(t) =\bigg\{ \begin{array}{ll}tV_\alpha\alpha bR\rho^t +\alpha bR\rho^t &\textrm{if}~  1-C\alpha=\rho
\\
 \frac{\alpha bRV_\alpha(1-C\alpha)}{1-C\alpha-\rho}\Big[(1-C\alpha)^t-\rho^t\Big] +\alpha bR\rho^t &\textrm{if}~ 1-C\alpha \neq \rho.
 \end{array}
\end{equation*} 
The proof is completed.
\end{proof}

\section{Estimate for the distance between $w^\alpha$ and $x_*$}\label{sec-5}

In the previous section, we studied the convergence property of the sequence $w (t)$ of the algorithm \eqref{eq-1-3} towards the unique fixed point $w^\alpha$ of $T_\alpha$.
This section is devoted to estimate the distance between the fixed point $w^\alpha$ and the minimizer $x_*$ of the aggregate cost $f$. Namely, we prove the result of Theorem \ref{thm-2-6}. Also we give the proof of Theorem \ref{thm-1-2} at the end of this section.

For our purpose, we consider the average $\bar{w}^\alpha $ of $w^\alpha=((w^\alpha_1)^\top, \cdots, (w^\alpha_n)^\top)^\top \in \mathbb{R}^{nd}$ defined by
\begin{equation*}
\bar{w}^\alpha=\frac{1}{n}\sum_{k=1}^nw_k^\alpha,
\end{equation*}
and note that
\begin{equation*}
W^\infty = \pi 1_n^\top
\end{equation*}
which leads to
\begin{equation} \label{eq-5-1}
(W^{\infty} \otimes I_d)w^\alpha=(\pi 1_n^\top  \otimes I_d)w^\alpha = n\pi \otimes \bar{w}^\alpha.
\end{equation}
We will find several estimates regarding the two quantities $\Vert w^\alpha-n\pi \otimes \bar{w}^\alpha \Vert_{\pi\otimes 1_d}$ and $\Vert \bar{w}^\alpha- x_* \Vert$ to prove Theorem \ref{thm-2-6}. In what follows, we consider the two cases and the value $\alpha_0$ defined in Theorem \ref{thm-2-3}.
\begin{lem} \label{lem-5-1}For $\alpha \in (0,\alpha_0]$, the fixed point $w^\alpha$ of $T_{\alpha}$ satisfies
\begin{equation}\label{eq-5-10}
\Vert w^\alpha-n\pi \otimes x_* \Vert_{\pi\otimes 1_d} \leq \Vert w^\alpha-n\pi \otimes \bar{w}^\alpha \Vert_{\pi\otimes 1_d} +n \Vert \bar{w}^\alpha- x_* \Vert.
\end{equation}
\end{lem}
\begin{proof}
We apply the triangle inequality to find
\begin{equation} \label{eq-5-2}
\Vert w^\alpha-n\pi \otimes x_* \Vert_{\pi\otimes 1_d} \leq \Vert w^\alpha-n\pi \otimes \bar{w}^\alpha \Vert_{\pi\otimes 1_d} + \Vert n\pi \otimes (\bar{w}^\alpha-x_*) \Vert_{\pi\otimes 1_d}.
\end{equation}
Using the stochastic property of $\pi$, we get 
\begin{equation*}
\Vert n\pi \otimes (\bar{w}^\alpha-x_*) \Vert_{\pi\otimes 1_d}^2 = \sum_{j=1}^n \frac{1}{\pi_j}(n \pi_j)^2 \Vert \bar{w}^\alpha-x_*\Vert^2=n^2 \Vert \bar{w}^\alpha- x_* \Vert^2.
\end{equation*}
Combining this with (\ref{eq-5-2}) gives
\begin{equation*}
\Vert w^\alpha-n\pi \otimes x_* \Vert_{\pi\otimes 1_d} \leq \Vert w^\alpha-n\pi \otimes \bar{w}^\alpha \Vert_{\pi\otimes 1_d} +n \Vert \bar{w}^\alpha- x_* \Vert.
\end{equation*}
The proof is done.
\end{proof} 
 We state a basic lemma.
\begin{lem} \label{lem-5-2}
For any $n \times n$ matrix $A$, we have
\begin{equation*}
||| A \otimes I_d |||_{\pi\otimes 1_d} =||| A |||_\pi .
\end{equation*}
\end{lem}
In the following lemma, we obtain an estimate of $\| \bar{w}^{\alpha} -x_*\|$ using the strong convexity of the aggregate cost $f$. 
\begin{lem} \label{lem-5-3}For $\alpha \in (0,\alpha_0]$, the fixed point $w^\alpha$ of $T_{\alpha}$ satisfies
\begin{equation}\label{eq-5-9}
\| \bar{w}^{\alpha} -x_*\| \leq  \frac{L}{\gamma}\sum_{i=1}^n \Big\| \bar{w}^{\alpha}- \frac{w_i^{\alpha}}{n\pi_i}\Big\|
\end{equation}
where $\gamma = \frac{\mu \bar{L}}{\mu+\bar{L}}$ with $\bar{L} =\frac{1}{n} \sum_{k=1}^n L_k$.
\end{lem}
\begin{proof}The fixed point equation $w^\alpha = T_{\alpha} (w^\alpha)$ can be written as
 \begin{equation*}
 w_k^{\alpha} = \sum_{j=1}^n W_{kj}\Big( w_{j}^{\alpha} - \alpha \nabla f_j \Big( \frac{w_j^{\alpha}}{n\pi_j}\Big)\Big).
 \end{equation*}
Averaging this with the fact that $W$ is column-stochastic, we get
\begin{equation*}
\bar{w}^{\alpha} = \bar{w}^{\alpha} - \frac{\alpha}{n} \sum_{i=1}^n \nabla f_i \Big( \frac{w_i^{\alpha}}{n\pi_i}\Big).
\end{equation*}
We manipulate this as follows
\begin{equation*}
\bar{w}^{\alpha} -x_* = \bar{w}^{\alpha} -x_* - \frac{\alpha}{n}    \sum_{i=1}^n \nabla f_i \big( \bar{w}^{\alpha}\big) + \frac{\alpha}{n} \sum_{i=1}^n \Big( \nabla f_i \big( \bar{w}^{\alpha}\big) - \nabla f_i \Big( \frac{w_i^{\alpha}}{n\pi_i}\Big) \Big).
\end{equation*}
Using the triangle inequality with Lemma \ref{lem-3-30}, we have
\begin{equation*}
\| \bar{w}^{\alpha} -x_*\| \leq \Big( 1- \frac{\gamma \alpha}{n}\Big) \| \bar{w}^{\alpha} -x_* \| + \frac{L \alpha}{n} \sum_{i=1}^n \Big\| \bar{w}^{\alpha}- \frac{w_i^{\alpha}}{n\pi_i}\Big\|,
\end{equation*}
which gives
\begin{equation*}
\| \bar{w}^{\alpha} -x_*\| \leq  \frac{L}{\gamma} \sum_{i=1}^n \Big\| \bar{w}^{\alpha}- \frac{w_i^{\alpha}}{n\pi_i}\Big\|.
\end{equation*}
The proof is done.
\end{proof}
We further estimate the right hand side of the inequality \eqref{eq-5-9} in the following lemma.
\begin{lem}\label{lem-5-4}For $\alpha \in (0,\alpha_0]$, the fixed point $w^\alpha$ of $T_{\alpha}$ satisfies
\begin{equation*}
\sum_{i=1}^n \Big\| \bar{w}^{\alpha}- \frac{w_i^{\alpha}}{n\pi_i}\Big\| \leq \frac{\sqrt{\sum_{k=1}^n\frac{1}{\pi_k}}}{n}\Vert w^\alpha-(W^{\infty} \otimes I_d)w^\alpha\Vert_{\pi\otimes 1_d}.
\end{equation*}
\end{lem}
\begin{proof}
Using the Cauchy-Schwartz inequality, we estimate
\begin{equation*}
\begin{split}
\sum_{i=1}^n \Big\| \bar{w}^{\alpha}- \frac{w_i^{\alpha}}{n\pi_i}\Big\| &=\frac{1}{n}\sum_{i=1}^n\frac{1}{\pi_i} \Big\|  w_i^\alpha -n\pi_i\bar{w}^{\alpha}\Big\| 
\\
&\leq \frac{1}{n}\Big(\sum_{i=1}^n\frac{1}{\pi_i}\Big)^{\frac{1}{2}}\Big( \sum_{i=1}^n\frac{1}{\pi_i}\Big\| w_i^\alpha -n\pi_i\bar{w}^{\alpha}\Big\|^2\Big)^{\frac{1}{2}}.
\end{split}
\end{equation*}
 By the definition of weighted ($\pi\otimes 1_d$)-norm and \eqref{eq-5-1}, we get
\begin{equation*}
\begin{split}
\sum_{i=1}^n \Big\| \bar{w}^{\alpha}- \frac{w_i^{\alpha}}{n\pi_i}\Big\| &\leq \frac{\sqrt{\sum_{k=1}^n\frac{1}{\pi_k}}}{n}\Vert w^\alpha-n\pi\otimes \bar{w}^{\alpha}\Vert_{\pi\otimes 1_d}
\\
&= \frac{\sqrt{\sum_{k=1}^n\frac{1}{\pi_k}}}{n}\Vert w^\alpha-(W^{\infty} \otimes I_d)w^\alpha\Vert_{\pi\otimes 1_d}.
\end{split}
\end{equation*}
The proof is done.
\end{proof}
Recall that we denote $\nabla F(\frac{w}{n\pi}) = (\nabla f_1 (\frac{w_1}{n\pi_1})^\top , \cdots, \nabla f_n (\frac{w_n}{n\pi_n})^\top )^\top $. In the following we will find a bound of the weighted consensus error $\Vert w^\alpha-(W^{\infty} \otimes I_d)w^\alpha\Vert_{\pi\otimes 1_d}$.

\begin{prop}\label{pro-5-5}For $\alpha \in (0,\alpha_0]$, the fixed point $w^\alpha$ of $T_{\alpha}$ satisfies
\begin{equation*}
\Vert w^\alpha-(W^{\infty} \otimes I_d)w^\alpha\Vert_{\pi\otimes 1_d} \leq  \frac{\alpha\rho}{1-\rho}\Big[ \frac{LR}{n\pi_{min}} +Q\Big]
\end{equation*} 
\end{prop}
where $\pi_{min} = \min_{1\leq j\leq n}\pi_j$ and $Q=\|\nabla F(0)\|_{\pi\otimes 1_d}$.
\begin{proof}
Using the fact that $w_\alpha$ is the fixed point of $T_\alpha$, we proceed as follows: 
\begin{equation}\label{eq-5-3}
\begin{split}
&\Vert w^\alpha-(W^{\infty} \otimes I_d)w^\alpha\Vert_{\pi\otimes 1_d}
\\
&=\Vert((I_n-W^{\infty})\otimes I_d)w^\alpha\Vert_{\pi\otimes 1_d}
\\
 &=\Big\|[((I_n-W^{\infty})W)\otimes I_d]\Big(w^\alpha-\alpha \nabla F\Big(\frac{w^\alpha}{n\pi}\Big)\Big)\Big\|_{\pi\otimes 1_d}
\\
&=\Big\|[(W-W^{\infty})\otimes I_d]\Big[((I_n-W^{\infty})\otimes I_d)w^\alpha-\alpha \nabla F\Big(\frac{w^\alpha}{n\pi}\Big)\Big]\Big\|_{\pi\otimes 1_d},
\end{split}
\end{equation}
where we have used $W^{\infty}W=WW^{\infty}=(W^{\infty})^2=W^{\infty}$ in the second and the third equalities. By the triangle inequality and Lemma \ref{lem-5-2}, we estimate \eqref{eq-5-3} as follows:
\begin{equation*}
\begin{split}
&\Vert w^\alpha-(W^{\infty} \otimes I_d)w^\alpha\Vert_{\pi\otimes 1_d}
\\
&\leq |||(W-W^{\infty})\otimes I_d|||_{\pi\otimes 1_d}\Big[\Vert((I_n-W^{\infty})\otimes I_d)w^\alpha\Vert_{\pi\otimes 1_d}+\alpha\Big\| \nabla F\Big(\frac{w^\alpha}{n\pi}\Big)\Big\|_{\pi\otimes 1_d}\Big]
\\
&=|||W-W^{\infty}|||_\pi\Big[\Vert w^\alpha-(W^{\infty} \otimes I_d)w^\alpha\Vert_{\pi\otimes 1_d} +\alpha\Big\| \nabla F\Big(\frac{w^\alpha}{n\pi}\Big)\Big\|_{\pi\otimes 1_d}\Big]
\\
&= \rho \Vert w^\alpha-(W^{\infty} \otimes I_d)w^\alpha\Vert_{\pi\otimes 1_d} +\alpha  \rho \Big\| \nabla F\Big(\frac{w^\alpha}{n\pi}\Big)\Big\|_{\pi\otimes 1_d}
\end{split}
\end{equation*} 
which leads to
\begin{equation}\label{eq-5-6}
\Vert w^\alpha-(W^{\infty} \otimes I_d)w^\alpha\Vert_{\pi\otimes 1_d} \leq \frac{ \alpha\rho}{1-\rho}\Big\| \nabla F\Big(\frac{w^\alpha}{n\pi}\Big)\Big\|_{\pi\otimes 1_d}.
\end{equation} 
Considering the $L$-smoothness of $f_i$ we find
\begin{equation*}
\begin{split}
\Big\| \nabla f_i \Big( \frac{w_i^{\alpha}}{n\pi_i}\Big)\Big\|& \leq \Big\|\nabla f_i \Big( \frac{w_i^{\alpha}}{n\pi_i}\Big) - \nabla f_i (0) \Big\| + \|\nabla f_i (0)\|
\\
&\leq L\Big\| \frac{w_i^{\alpha}}{n\pi_i} \Big\| + \|\nabla f_i (0)\|.
\end{split}
\end{equation*}
Using this we have
\begin{equation*}
\begin{split}
\Big\|\nabla F\Big(\frac{w_{\alpha}}{n\pi}\Big)\Big\|_{\pi \otimes 1_d} & = \Big( \sum_{j=1}^n \frac{1}{\pi_j}\Big\|\nabla f_j \Big( \frac{w_j^{\alpha}}{n\pi_j}\Big) \Big\|^2\Big)^{\frac{1}{2}}
\\
& \leq \Big( \sum_{j=1}^n \frac{1}{\pi_j} \Big( L^2 \frac{\|w_j^{\alpha}\|^2}{(n\pi_j)^2} \Big) \Big)^{\frac{1}{2}} + \Big( \sum_{j=1}^n \frac{1}{\pi_j}\|\nabla f_j (0)\|^2 \Big)^{\frac{1}{2}}
\end{split}
\end{equation*}
Applying Theorem \ref{thm-2-4} we get
\begin{equation*}
\begin{split}
\Big\|\nabla F\Big(\frac{w_{\alpha}}{n\pi}\Big)\Big\|_{\pi \otimes 1_d}& \leq \frac{L}{n\pi_{min}}\Big( \sum_{j=1}^n \frac{1}{\pi_j} \|w_j^{\alpha}\|^2 \Big)^{\frac{1}{2}} + Q
\\
&= \frac{L}{n\pi_{min}}\|w^\alpha\|_{\pi\otimes 1_d} +Q
\\
& \leq \frac{L}{n\pi_{min}}R +Q
\end{split}
\end{equation*}
where $Q=\Big( \sum_{j=1}^n \frac{1}{\pi_j}\|\nabla f_j (0)\|^2 \Big)^{\frac{1}{2}}$.
Putting this estimate in \eqref{eq-5-6}, we have
\begin{equation} \label{eq-5-4}
\Vert w^\alpha-(W^{\infty} \otimes I_d)w^\alpha\Vert_{\pi\otimes 1_d} \leq \frac{\alpha\rho}{1-\rho}\Big( \frac{LR}{n\pi_{min}}+Q\Big).
\end{equation} 
The proof is done.
\end{proof} 

Combining the above lemmas, we proceed to prove the result of Theorem \ref{thm-2-6}.
\begin{proof}[Proof of Theorem \ref{thm-2-6}]
Applying Lemma \ref{lem-5-4} to Lemma \ref{lem-5-3}, we have
\begin{equation}\label{eq-5-5}
\|\bar{w}^{\alpha} -x_*\| \leq \frac{L}{n\gamma}\sqrt{\sum_{k=1}^n\frac{1}{\pi_k}}\Vert w^\alpha-(W^{\infty} \otimes I_d)w^\alpha\Vert_{\pi\otimes 1_d}.
\end{equation} 
Using the relation \eqref{eq-5-1} in the inequality \eqref{eq-5-10} we find
\begin{equation*}
\Vert w^\alpha-n\pi \otimes x_* \Vert_{\pi\otimes 1_d} \leq \Vert w^\alpha-(W^{\infty} \otimes I_d)w^\alpha\Vert_{\pi\otimes 1_d} +n \Vert \bar{w}^\alpha- x_* \Vert.
\end{equation*}
Inserting \eqref{eq-5-5} and Proposition \ref{pro-5-5} to the above inequality, we get
\begin{equation*}
\begin{split}
\Vert w^\alpha-n\pi \otimes x_* \Vert_{\pi\otimes 1_d}  &\leq \Big(1+\frac{L}{\gamma}\sqrt{\sum_{j=1}^n \frac{1}{\pi_j}}\Big)\Vert w^\alpha-(W^{\infty} \otimes I_d)w^\alpha\Vert_{\pi\otimes 1_d}
\\
& \leq \frac{\alpha \rho}{1-\rho}\Big(1+\frac{L}{\gamma}\sqrt{\sum_{j=1}^n \frac{1}{\pi_j}}\Big)\Big( \frac{LR}{n\pi_{min}}+Q\Big).
\end{split}
\end{equation*}
The proof is done.
\end{proof}

Now we are ready to prove Theorem \ref{thm-1-2}.
\begin{proof}[Proof of Theorem \ref{thm-1-2}]Consider the two cases and the value $\alpha_0 >0$ of Theorem \ref{thm-2-3}. 
Consider the sequence $w(t) \in \mathbb{R}^{nd}$ of the algorithm \eqref{eq-1-3} with stepsize $\alpha >0$ with initial point $w(0)$.  Then, the result of Theorem \ref{thm-2-5} gives
\begin{equation*} 
\Vert w(t+1)-w^\alpha\Vert_{\pi\otimes1_d} \leq V_\alpha (1-C\alpha)^{t+1}\Vert w(0)-w^\alpha\Vert_{\pi\otimes 1_d} + \mathcal{R}(t).
\end{equation*} 
Combining this with the result of Theorem \ref{thm-2-6} and the triangle inequality gives
\begin{equation*}
\begin{split}
\Vert w (t+1)-n\pi \otimes x_* \Vert_{\pi\otimes 1_d}  \leq&  V_\alpha (1-C\alpha)^{t+1}\Vert w(0)-w^\alpha\Vert_{\pi\otimes 1_d} + \mathcal{R}(t)
\\
&\quad + \frac{\alpha \rho}{1-\rho}\Big(1+\frac{L}{\gamma}\sqrt{\sum_{j=1}^n \frac{1}{\pi_j}}\Big)\Big( \frac{LR}{n\pi_{min}}+Q\Big).
\end{split}
\end{equation*}
Therefore we have
\begin{equation*}
\lim_{t \rightarrow \infty} \Vert w  (t+1)-n\pi \otimes x_* \Vert_{\pi\otimes 1_d}   \leq \frac{\alpha \rho}{1-\rho}\Big(1+\frac{L}{\gamma}\sqrt{\sum_{j=1}^n \frac{1}{\pi_j}}\Big)\Big( \frac{LR}{n\pi_{min}}+Q\Big),
\end{equation*}
which implies that for each $1\leq k \leq n$, the sequence $\{ w_k (t)\}_{t \geq 0}$ converges  to an $O(\alpha)$-neighborhood of $n\pi_k x_*$ exponentially fast. We recall from \eqref{eq-1-2} and \eqref{eq-1-6} that $z_k (t) = w_k (t)/y_k (t)$ and    $y_k (t)$ converges to $n\pi_k$ exponentially fast. Therefore, the sequence $\{z_k (t)\}_{t \geq 0}$ converges to an $O(\alpha)$-neighborhood of  $x_*$ linearly as $t \rightarrow \infty$. The proof is done.
\end{proof}

\section{Simulation}\label{sec-6}
In this section, we provide numerical experiments for the gradient-push algorithm \eqref{eq-1-2} that support the   convergence results given in this paper. Let $w_{\alpha}(t) \in \mathbb{R}^{nd}$ and $z_k (t)$ are variables of the gradient-push algorithm \eqref{eq-1-2} (see also \eqref{eq-1-3}). We also denote by $w^{\alpha}\in \mathbb{R}^{nd}$ the fixed point of the operator $T^{\alpha}$. Then we aim to provide numerical results to support the following results of this paper: 
\begin{itemize}
\item Theorem \ref{thm-2-3}: Contraction property of $T_{\alpha}$ for $\alpha \in (0,\alpha_0]$. 
\item Theorem \ref{thm-2-5}: $\lim_{t \rightarrow 0} \|w (t) - w^{\alpha}\|_{\pi \otimes 1_d} =0$ for each stepsize $\alpha \in (0,\alpha_0]$.
\item Theorem \ref{thm-2-6}: $\|w^{\alpha} - n\pi \otimes x_*\|_{\pi \otimes 1_d} = O(\alpha)$ for $\alpha \in (0,\alpha_0]$.
\item Theorem \ref{thm-1-2}: $\lim_{t \rightarrow \infty}\sum_{k=1}^{n} \|z_k (t) -x_*\| =O(\alpha)$ for each stepsize $\alpha \in (0,\alpha_0]$.
\end{itemize}
We set $n=20$ and the mixing matrix $W$ is generated by a random directed graph $G$ with $n$ vertices by letting each arc to occur with probability $p$ = 0.7, except for the self-loops, which must exist with probability $1$.

\medskip 

\noindent \textbf{Settings for the numerical experiment.}
We choose the local costs functions satisfying one of two conditions given in the main results:
 \begin{enumerate}
\item[Case 1:] Each $f_j$ is $\mu_j$-strongly convex and $L_j$-smooth for $1 \leq j \leq n$. The $\alpha_0 >0$ is set to be  $\alpha_0 =\min_{1 \leq j \leq n} \frac{2n\pi_j}{L_j +\mu_j}$. For this we consider the regularized least square problem:
\begin{equation}\label{eq-6-1}
\min_{x \in \mathbb{R}^{n}}\Big\{f(x) := \frac{1}{2n}\sum_{j=1}^n \Big(\Vert A_j x - b_j \Vert^2 +\delta \|x\|^2\Big)\Big\},
\end{equation}
where $A_j \in \mathbb{R}^{m\times d}$ and $b_j \in \mathbb{R}^{m}$ with $d=3$, $m=4$ and $\delta =2$. The value $\alpha_0$ is cimputed as  $\alpha_0=\min_{1 \leq k \leq n} \frac{2n\pi_k}{L_k +\mu_k}$ where $L_k=\max(\lambda(A_k^\top A_k+\delta I_d))$ and $\mu_k=\min(\lambda(A_k^\top A_k+\delta I_d))$. Each element of the matrix $A_j$ and the vector $b_j$ for $1 \leq j\leq n$ is randomly generated by the uniform distribution on $[0,1]$.
\item[Case 2:] Each cost $f_j$ is convex and $L_j$-smooth while the aggregate cost $f$= $\frac{1}{n}\sum_{j=1}^n f_j$ is $\mu$-strongly convex. The $\alpha_0 >0$ is defined as  $\alpha_0 = \min_{1 \leq k \leq n}\frac{2n\pi_k}{L_k+\epsilon}$ for a fixed $\ep>0$. We consider local cost given by a convex quadratic cost as follows:
\begin{equation}\label{eq-6-2}
\min_{x \in \mathbb{R}^{n}}\Big\{f(x) := \frac{1}{n}\sum_{j=1}^n \Big(\frac{1}{2} x^\top A_jx+b_jx \Big)\Big\},
\end{equation}
where $A_j \in \mathbb{R}^{d\times d}$ and $b_j \in \mathbb{R}^d$ with $d=10$. We generate $A_j$ as $A_j = R_jR_j^\top $ where each $R_j$ is $d \times m$ matrix with $m <d$ whose elements are randomly generated by the uniform distribution on  $(0,1)$. The elements of $b_j$ are generated randomly by the uniform distribution on $(0,1)$. The value $\alpha_0$ is computed as  $\alpha_0 = \min_{1 \leq k \leq n}\frac{2n\pi_k}{L_k+\epsilon}$ where $L_k=\max(\lambda(A_k))$ and $\epsilon =0.01$.
\end{enumerate}

\subsection{Contraction property of the mapping $T_{\alpha}$}
First we validate numerically the contraction property of $T_{\alpha}$ for $\alpha \in (0,\alpha_0]$ proved in Theorem \ref{thm-2-3}. We mention that for both Case 1 and Case 2, the local costs are given by the form $f_j(x) = \frac{1}{2}x^\top P_j x+q_j x$(ignore the constant part) for $1 \leq j \leq n$ with $P_j$ and $q_j$ given by
\begin{equation*}
(P_j, q_j) =\bigg\{\begin{array}{ll} (A_j^\top A_j + \delta I_d, - 2 A_j^\top b_j)&\textrm{for Case 1}
\\
(A_j, b_j) &\textrm{for Case 2}.
\end{array} 
\end{equation*}
and so  the operator $T_{\alpha}$ satisfies
\begin{equation*}
\begin{split}
&T_{\alpha}(x)- T_{\alpha}(y) 
\\
& = \Big( \sum_{j=1}^n W_{1j} (I-\alpha P_j) (x_j -y_j)^\top, \cdots, \sum_{j=1}^n W_{nj} (I-\alpha P_j) (x_j -y_j)^\top \Big)^\top
\\
&=:M_{\alpha}(x-y),
\end{split}
\end{equation*}
where $M_{\alpha}$ is an $nd\times nd$ matrix whose $(i,j)$-th block is given by the $d\times d$-matrix $W_{ij}(I-\alpha P_j)$ for $1\leq i, j \leq n$. The lipschitz constant $\mathcal{L}_{\alpha}$ of $T_{\alpha}$ is then computed as follows:
\begin{equation*}
\begin{split}
\mathcal{L}_{\alpha} &= \sup_{x \neq y} \frac{\|T_{\alpha}(x) -T_{\alpha}(y)\|_{\pi \otimes 1_d}}{\|x-y\|_{\pi \otimes 1_d}} = \sup_{x \neq y} \frac{\|M_{\alpha}(x-y)\|_{\pi \otimes 1_d}}{\|x-y\|_{\pi \otimes 1_d}}
\\
&=\sup_{x \in \mathbb{R}^{nd}\setminus \{0\}} \frac{\|D^{-1}M_{\alpha} x\|_{L^2 (\mathbb{R}^{nd})}}{\|D^{-1}x\|_{L^2 (\mathbb{R}^{nd})}},
\end{split}
\end{equation*}
where $D$ is an $nd\times nd$ matrix given by $ \textrm{Diag}(\sqrt{\pi_1}, \cdots, \sqrt{\pi_n})\otimes I_d$.
Letting $x=Dy$, we estimate
\begin{equation*}
\begin{split}
\mathcal{L}_{\alpha} &= \sup_{y \in \mathbb{R}^{nd}\setminus \{0\}} \frac{\|D^{-1} M_{\alpha} D  y\|_{L^2 (\mathbb{R}^{nd})}}{\|y\|_{L^2 (\mathbb{R}^{nd})}}
\\
& = ||| D^{-1} M_{\alpha} D |||_{L^2 \rightarrow L^2},
\end{split}
\end{equation*} We compute the value $\mathcal{L}_{\alpha}=||| D^{-1} M_{\alpha} D|||_{L^2 \rightarrow L^2}$ by applying the command \emph{norm} in Matlab. The graphs of $\mathcal{L}_{\alpha}$ with respect to $\alpha \in (0,2\alpha_0]$ are presented in Figure \ref{fig2} for Case 1 and Case 2. The results show that $\mathcal{L}_{\alpha} \leq 1- C\alpha$ for $\alpha \in (0,\alpha_0]$ with a value $C>0$ which confirms the contraction property \eqref{eq-2-11} of \mbox{Theorem \ref{thm-2-3}.}
\begin{figure}[htbp]
\includegraphics[height=5cm, width=7.3cm]{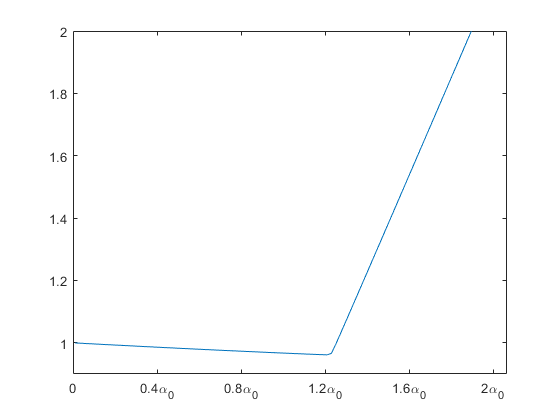}
\includegraphics[height=5cm, width=7.3cm]{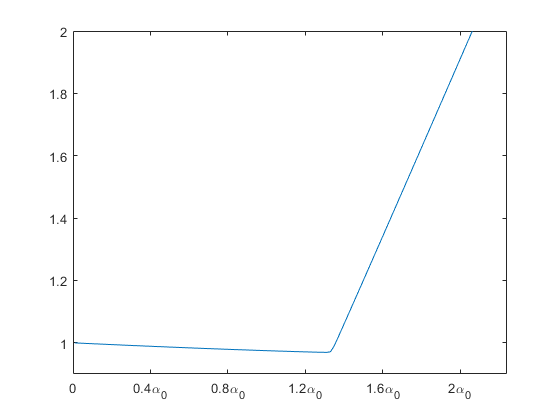}
\vspace{-0.3cm}\caption{The graph of $\mathcal{L}_{\alpha}$ with respect to $\alpha \in (0,2\alpha_0]$. Left: Case 1. Right: Case 2.}\label{fig2}
\end{figure}

For each $\alpha \in (0,\alpha_0]$, we compute the fixed point $w^{\alpha}$ of $T_{\alpha}$ numerically using the property that $\lim_{t \rightarrow \infty}(T_{\alpha})^t (w(0)) = w^{\alpha}$ to validate Theorem \ref{thm-2-5} and Theorem \ref{thm-2-6} in the below.


\subsection{Convergence results for the strongly convex  case (Case 1)}\label{subsec-6-1}

We compute the minimizer $x_*$ of \eqref{eq-6-1} by the following formula 
\begin{equation*}
x_* = \Big\{\sum_{j=1}^n \Big(A_j ^\top A_j + \delta I_d\Big)\Big\}^{-1}\sum_{j=1}^n A_j ^\top b_j.
\end{equation*}
Figure \ref{fig3} represents the graph of $\log_{10}\Big(\Vert w_{\alpha}(t) - w^\alpha \Vert_{\pi\otimes1_d}\Big)$ with respect to the iteration $t \geq 0$ for the case $\alpha = \alpha_0$ and the graph of $\Vert w^\alpha-n\pi \otimes x_* \Vert_{\pi\otimes 1_d}$ for stepsizes $\alpha \in (0,\alpha_0]$.
\begin{figure}[htbp]
\includegraphics[height=5cm, width=7.3cm]{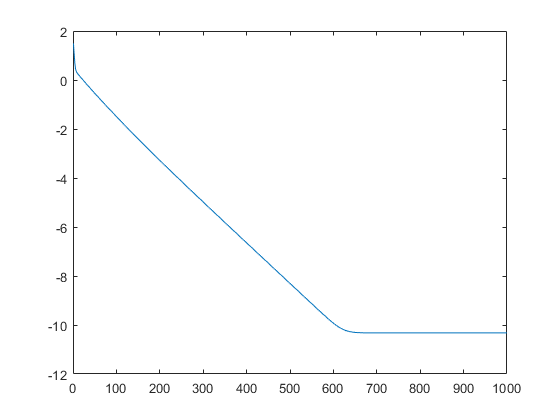} 
\includegraphics[height=5cm, width=7.3cm]{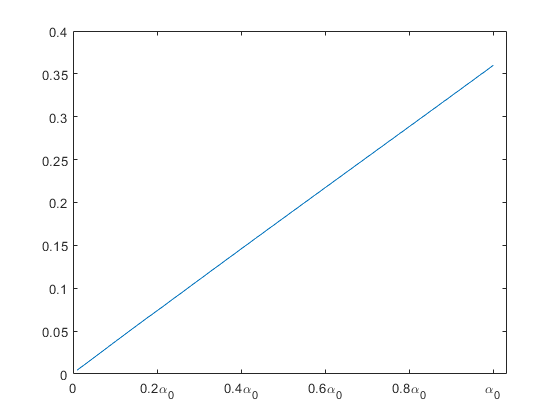} 
\vspace{-0.3cm}\caption{Left: The graph of $\log_{10}\Vert w_{\alpha}(t) - w^\alpha \Vert_{\pi\otimes1_d}$ for least square problem with stepsize $\alpha$ = $\alpha_0$, itertation $t=1,\cdots,1000$. Right: The graph  of the error $\Vert w^\alpha-n\pi \otimes x_* \Vert_{\pi\otimes 1_d}$ for stepsizes $\alpha \in (0,\alpha_0]$.}\label{fig3}
\end{figure}
As we expected in Theorem \ref{thm-2-5},  the distance $\Vert w_{\alpha}(t) - w^\alpha \Vert_{\pi\otimes1_d}$ converges to $0$ exponentially fast(up to the machine error in the experiment) and the distance $\Vert w^\alpha-n\pi \otimes x_* \Vert_{\pi\otimes 1_d}$  between minimizer $x_*$ and the fixed point $w^\alpha$ is $O(\alpha)$, which coincides with the result of Theorem \ref{thm-2-6}.

Lastly, we test the convergence property of the gradient-push algorithm \eqref{eq-1-2} towards the minimizer $x_*$ of the global cost $f$
with stepsizes
\begin{equation*} 
\alpha \in \Big\{ 0.2 \alpha_0,  0.5 \alpha_0,  \alpha_0, 1.3 \alpha_0 \Big\}.
\end{equation*}
 Figure \ref{fig4} shows the graph of $ \|w_\alpha (t) -x_*\|_{\pi\otimes 1_d}$ with respect to $t \in \mathbb{N}$. The result shows that the gradient-push algorithm with stepsize $\alpha \in \{0.2\alpha_0, 0.5\alpha_0, \alpha_0\}$ converges linearly to an $O(\alpha)$-neighborhood of the minimizer  $x_*$ as expected by Theorem \ref{thm-1-2} while the algorithm with stepsize $\alpha = 1.3\alpha_0$ diverges.
\begin{figure}[htbp]
\includegraphics[height=5cm, width=7.3cm]{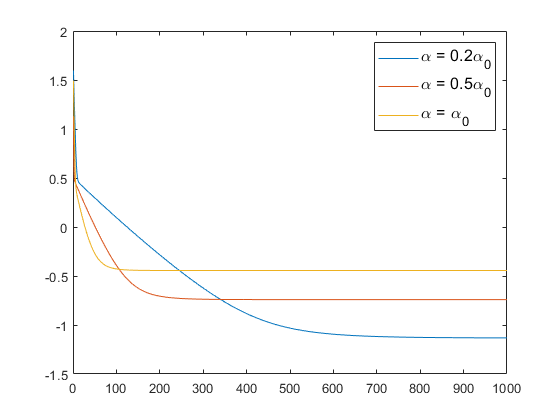}
\includegraphics[height=5cm, width=7.3cm]{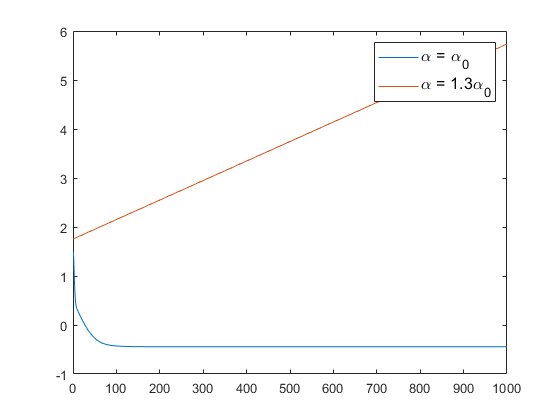}  
\vspace{-0.3cm}\caption{Both graphs show  the $\log_{10}\|w_{\alpha}(t)  - n\pi \otimes x_*\|_{\pi \otimes 1_d}$ for regularized least square problem with various stepsizes. Left: Cases of the convergent results.  Right: Case of the divergent result.}\label{fig4}
\end{figure}

\subsection{Convergence results for the   convex  case (Case 2)}
The minimizer of \eqref{eq-6-2} is given by
\begin{equation*}
x_* =- \Big(\sum_{j=1}^n A_j\Big)^{-1}\sum_{j=1}^n b_j^\top .
\end{equation*} 

Figure \ref{fig5} presents the graph of $\log_{10}\Big(\Vert w_{\alpha}(t) - w^\alpha \Vert_{\pi\otimes1_d}\Big)$ for $\alpha = \alpha_0$ with respect to the iteration $t\geq 0$ and the graph of $\Vert w^\alpha-n\pi \otimes x_* \Vert_{\pi\otimes 1_d}$ with stepsizes $\alpha \in (0,\alpha_0]$. The measure $\Vert w_{\alpha}(t) - w^\alpha \Vert_{\pi\otimes1_d}$ converges to $0$(up to the machine error in the experiment) for $\alpha =\alpha_0$ as predicted by Theorem \ref{thm-2-5}. Also the distance $\Vert w^\alpha-n\pi \otimes x_* \Vert_{\pi\otimes 1_d}$ is $O(\alpha)$, as expected in the result of  Theorem \ref{thm-2-6}.
\begin{figure}[htbp]
\includegraphics[height=5cm, width=7.3cm]{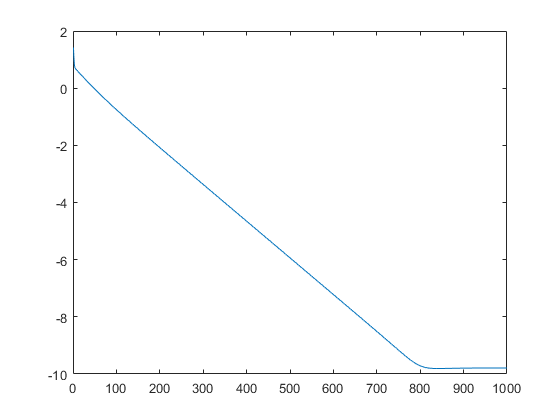} 
\includegraphics[height=5cm, width=7.3cm]{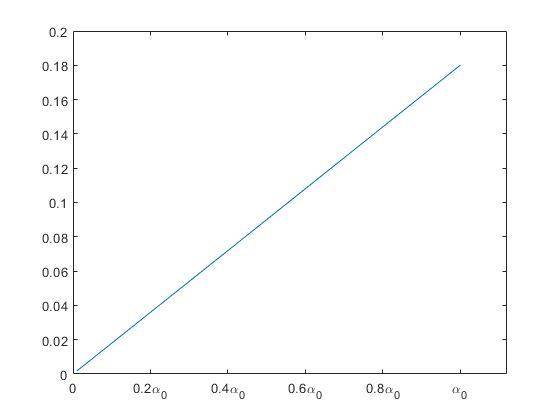} 
\vspace{-0.3cm}\caption{Left: The graphs of $ \Vert w_{\alpha}(t) - w^\alpha \Vert_{\pi\otimes1_d}$ for the convex quadratic problem with stepsize $\alpha$ = $\alpha_0$, itertation $t=1,\cdots,1000$. Right: The graphs of the distance $\Vert w^\alpha-n\pi \otimes x_* \Vert_{\pi\otimes 1_d}$ for stepsizes $\alpha \in (0,\alpha_0]$.}\label{fig5}
\end{figure}

Next we test the gradient-push algorithm with stepsize  $\alpha \in \Big\{ 0.2 \alpha_0,  0.5\alpha_0,    \alpha_0,  1.45\alpha_0 \Big\}$. The graphs of $  \|w_\alpha (t) -x_*\|_{\pi\otimes 1_d}$ with respect to $t \in \mathbb{N}$ are presented in Figure \ref{fig6}. It turns out that the algorithm with stepsize $\alpha \in \{0.2 \alpha_0, 0.5\alpha_0, \alpha_0\}$ converges to an $O(\alpha)$-neighborhood of the minimizer $x_*$ as predicted by Theorem \ref{thm-1-2} but the algorithm with stepsize $\alpha =1.45\alpha_0$ diverges.
\begin{figure}[htbp]
\includegraphics[height=5cm, width=7.3cm]{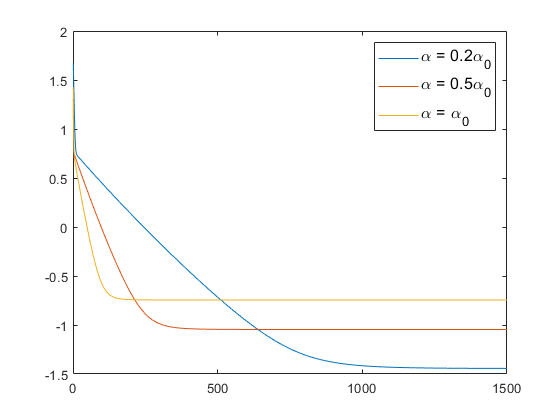}
\includegraphics[height=5cm, width=7.3cm]{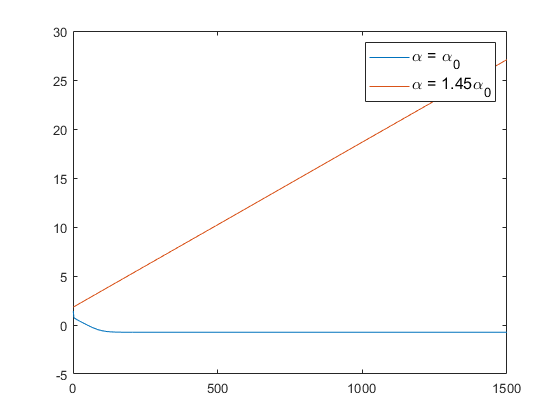}  
\vspace{-0.3cm}\caption{The graphs of $\log_{10}\big(\|w_{\alpha}(t)  - n\pi \otimes x_*\|_{\pi \otimes 1_d}\big)$ for convex quadratic problem with various stepsizes. Left: Cases of the convergent results. Right: Case of the divergent result.}\label{fig6}
\end{figure}


\end{document}